\documentclass{amsart}

\usepackage{amsmath, amssymb, enumerate}
\input xy
\xyoption{all}
\numberwithin{equation}{section}

\begin{document}

\newtheorem{theorem}{Theorem}[section]
\newtheorem{lemma}[theorem]{Lemma}
\newtheorem{corollary}[theorem]{Corollary}
\newtheorem{fact}[theorem]{Fact}
\newtheorem{proposition}[theorem]{Proposition}

\theoremstyle{definition}
\newtheorem{example}[theorem]{Example}
\newtheorem{definition}[theorem]{Definition}
\newtheorem{question}[theorem]{Question}

\theoremstyle{remark}
\newtheorem{claim}[theorem]{Claim}
\newtheorem{remark}[theorem]{Remark}

\def\endo{\operatorname{End}}
\def\fix{\operatorname{Fix}}
\def\ccma{\operatorname{CCMA}}
\def\dcf{\operatorname{DCF}}
\def\dcfa{\operatorname{DCFA}}
\def\acfa{\operatorname{ACFA}}
\def\ccm{\operatorname{CCM}}
\def\jet{\operatorname{Jet}}
\def\hom{\operatorname{Hom}}
\def\Th{\operatorname{Th}}
\def\cb{\operatorname{Cb}}
\def\cbsigma{\operatorname{Cb}_{\sigma}}
\def\qfcb{\operatorname{qfCb}}
\def\tp{\operatorname{tp}}
\def\tpsigma{\operatorname{tp}_{\sigma}}
\def\dimsigma{\operatorname{dim}_{\sigma}}
\def\qftp{\operatorname{qftp}}
\def\Stab{\operatorname{Stab}}
\def\stp{\operatorname{stp}}
\def\acl{\operatorname{acl}}
\def\aclsigma{\operatorname{acl}_{\sigma}}
\def\dcl{\operatorname{dcl}}
\def\dclsigma{\operatorname{dcl}_{\sigma}}
\def\eq{\operatorname{eq}}
\def\loc{\operatorname{CCM-loc}}
\def\jet{\operatorname{Jet}}
\def\alg{\operatorname{alg}}
\def\rank{\operatorname{rank}}
\def\SU{\operatorname{SU}}
\def\aut{\operatorname{Aut}}
\def\gal{\operatorname{Gal}}
\def\loc{\operatorname{loc}}

\def\ccmsigma{\ccm\text{-}\sigma}
\def\C{\mathcal C}
\renewcommand{\L}{{\mathcal L}}
\def\A{\mathcal A}
\def\X{\mathcal X}

\def\p{\bf P}

\def\PP{{\bf P}} 

\def\UU{\mathbb{U}}
\def\QQ{\mathbb{Q}}
\def\ZZ{\mathbb{Z}}
\def\NN{\mathbb{N}}
\def\CC{\mathbb{C}}

\newcommand{\spanof}[1]{\left< #1\right>}

\newcommand{\slantfrac}[2]{\hbox{$\,^{#1}\!/_{#2}$}}  
\newcommand{\quot}[2]{\slantfrac{#1}{#2}}

\newcommand{\isom}{\cong}
\newcommand{\elemequiv}{\equiv}
\newcommand{\elemeq}{\elemequiv}

\newcommand{\maps}{\rightarrow} 

\newcommand{\imp}{\rightarrow}
\providecommand{\s}{\sigma}
\providecommand{\id}{\operatorname{id}}

\renewcommand{\O}{\mathcal{O}}

\newcommand{\elex}{\preccurlyeq}
\newcommand{\elres}{\succcurlyeq}


\def\Ind#1#2{#1\setbox0=\hbox{$#1x$}\kern\wd0\hbox to 0pt{\hss$#1\mid$\hss}
\lower.9\ht0\hbox to 0pt{\hss$#1\smile$\hss}\kern\wd0}
\def\ind{\mathop{\mathpalette\Ind{}}}
\def\notind#1#2{#1\setbox0=\hbox{$#1x$}\kern\wd0\hbox to 0pt{\mathchardef
\nn=12854\hss$#1\nn$\kern1.4\wd0\hss}\hbox to
0pt{\hss$#1\mid$\hss}\lower.9\ht0 \hbox to
0pt{\hss$#1\smile$\hss}\kern\wd0}
\def\nind{\mathop{\mathpalette\notind{}}}

\date{\today}

\title[Compact complex manifolds with an automorphism]{Model theory of compact complex manifolds with an automorphism}
\author{Martin Bays}

\address{Martin Bays\\
\\
Institut f\"{u}r Logik und Grundlagenforschung\\
Fachbereich Mathematik und Informatik\\
Universit\"{a}t M\"{u}nster\\
Einsteinstrasse 62\\
48149 M\"{u}nster\\
Germany}
\email{mbays@sdf.org}

\author{Martin Hils}

\address{Martin Hils\\
Univ Paris Diderot, Sorbonne Paris Cit\'e, Institut de Math\'ematiques de Jussieu--Paris Rive Gauche, UMR 7586, CNRS, Sorbonne Universit\'es, UPMC Univ Paris 06, F-75013, Paris, France}

\curraddr{Institut f\"{u}r Logik und Grundlagenforschung\\
Fachbereich Mathematik und Informatik\\
Universit\"{a}t M\"{u}nster\\
Einsteinstrasse 62\\
48149 M\"{u}nster\\
Germany}

\email{hils@uni-muenster.de}

\thanks{M. Hils was partially funded by the Agence Nationale de Recherche [ValCoMo, Projet ANR blanc ANR-13-BS01-0006].}

\author{Rahim Moosa}

\address{Rahim Moosa\\
University of Waterloo\\
Department of Pure Mathematics\\
200 University Avenue West\\
Waterloo, Ontario \  N2L 3G1\\
Canada}
\email{rmoosa@uwaterloo.ca}

\thanks{R. Moosa was partially supported by an NSERC Discovery Grant.}

\begin{abstract}
Motivated by possible applications to meromorphic dynamics, and generalising known properties of difference-closed fields, this paper studies the theory $\ccma$ of compact complex manifolds with a generic automorphism.
It is shown that while $\ccma$ does admit geometric elimination of imaginaries, it cannot eliminate imaginaries outright: a counterexample to $3$-uniqueness in $\ccm$ is exhibited.
Finite-dimensional types are investigated and it is shown, following the approach of Pillay and Ziegler, that the canonical base property holds in $\ccma$.
As a consequence the Zilber dichotomy is deduced: finite-dimensional minimal types are either one-based or almost internal to the fixed field.
In addition, a general criterion for stable embeddedness in $TA$ (when it exists) is established, and used to determine the full induced structure of $\ccma$ on projective varieties, simple nonalgebraic complex tori, and simply connected nonalgebraic strongly minimal manifolds.

\end{abstract}

\maketitle
\setcounter{tocdepth}{1}
\tableofcontents

\section{Introduction}

\noindent
Underlying the applications of model theory to algebraic dynamics~\cite{descent,medvedevscanlon} is the equivalence of the category of rational dynamical systems with that of finitely generated difference fields:
To a rational dynamical system $(V,f)$ over $\mathbb C$, where $V$ is a projective algebraic variety and $f:V\to V$ is a dominant rational map, one associates the difference field $\big(\mathbb C(V),f^*\big)$ where $\mathbb C(V)$ is the rational function field and $f^*:\mathbb C(V)\to \mathbb C(V)$ is the endomorphism induced by precomposition with~$f$.
The theory of difference fields is tractable from a model-theoretic point of view because it admits a model companion, the theory $\acfa$ of difference-closed fields, which is tame: it admits a certain quantifier reduction, eliminates imaginaries, is supersimple, and realises a Zilber dichotomy for minimal types.

What happens when $(V,f)$ is a {\em meromorphic dynamical system}, that is, when $V$ is a possibly nonalgebraic compact complex manifold and $f$ is a dominant meromorphic map?
The first challenge is to find a replacement for the rational function field.
The meromorphic function field will not do as it only captures
the algebraic part of $V$.
For example, $\mathbb C(V)=\mathbb C$ when $V$ is a simple nonalgebraic complex torus.
An appropriate generalisation of the rational function field is
suggested by model theory.
Let $\A$ denote the structure whose sorts are (reduced and irreducible) compact complex analytic spaces and where the basic relations are all complex analytic subsets.
The first-order theory of $\A$, in this language $L$, is denoted by $\ccm$.
It is a stable theory that admits quantifier elimination, eliminates imaginaries, and whose sorts are of finite Morley rank --- see~\cite{survey} for a survey of this theory.
Let $\A'$ be an $\aleph_0$-saturated elementary extension of~$\A$.
Then $V$ is a sort and by saturation there exists a {\em generic} point of $V$ in $\A'$, i.e. an element of $V(\A')$ that it is not contained in $X(\A')$ for any proper complex analytic subset $X\subset V$.
What model theory suggests as an analogue to the rational function field is the $L$-substructure of $\A'$ whose underlying set is the definable closure of such a generic point.
Let us denote this $L$-structure, which is uniquely determined up to $L$-isomorphism, by $\C(V)$.
Note that by quantifier elimination, $\C(V)$ can be identified with the set of all dominant meromorphic maps $g:V\to W$, as $W$ ranges over all compact complex manifolds.
(In contrast, the meromorphic function field corresponds to considering only $W=\mathbb P(\mathbb C)$, the projective line.)

We can now imitate the difference-field construction.
If $(V,f)$ is a meromorphic dynamical system then $f$ induces a map $f^*:\C(V)\to\C(V)$ by precomposition.
Namely, fixing $c$ generic in $V$, every element of $\C(V)$ is of the form $g(c)$ for some dominant meromorphic map $g:V\to W$, and $f^*(g(c)):=g(f(c))$.
Note that as $f$ is dominant, $f(c)$ is again generic in $V(\A')$, and hence $g$ is defined at $f(c)$.
The dominance of $f$ also ensures that $f^*$ is injective.
It is not hard to check that $f^*$ is an $L$-embedding and, exactly extending the algebraic case, the association $(V,f)\mapsto\big(\C(V),f^*\big)$ is an equivalence of categories.

This motivates us to study substructures of models of $\ccm$ equipped with $L$-embeddings.
That is, letting $L_\sigma:=L\cup\{\sigma\}$, we are interested in the universal $L_\sigma$-theory
$\displaystyle \ccm_{\forall,\sigma}:=\ccm_\forall\cup\{\sigma\text{ is an $L$-embedding}\}$.
As noted in~\cite{bgh}, $\ccm_{\forall,\sigma}$ has a model companion, denoted by $\ccma$.
Various tameness properties for $\ccma$, in particular supersimplicity, follow automatically from the work of Chatzidakis and Pillay~\cite{ChPi98} on generic automorphisms.
In this paper we begin a finer study of $\ccma$.

It may be worth emphasising that actual compact complex
manifolds have no nontrivial $L$-automorphisms since every point is named
in the language.
It is the consideration of $L$-automorphisms of {\em nonstandard models}
of compact complex manifolds that gives rise to $\ccma$.

Our first observation about $\ccma$ is negative: while it admits geometric elimination of imaginaries,
\begin{itemize}
\item[]
{\em $\ccma$ does not eliminate imaginaries.} (Corollary~\ref{notEI})
\end{itemize}
To prove this we use a characterisation of Hrushovski's which reduces the problem to proving that $\ccm$ does not satisfy the property of {\em 3-uniqueness} introduced in~\cite{HrGroupoids}.
The failure of $3$-uniqueness in $\ccm$, which is what Section~\ref{not3unique} is mostly dedicated to, is of independent
interest because it distinguishes $\ccm$ from the stable theories that
arise in the model theory of fields
($\operatorname{ACF},\operatorname{DCF},\operatorname{SCF}$).

Since the complex field is definable in $\ccm$, on the projective line, in $\ccma$ we have a definable difference-closed field extending $(\mathbb C,+,\times,\id)$.
One of the things we point out here is that in this way
\begin{itemize}
\item[]
{\em $\acfa$ is purely stably embedded in $\ccma$.} (Theorem~\ref{purityalg})\footnote{Here, and throughout, by $\acfa$ we really mean the complete theory of models of $\acfa$ extending the trivial difference field $(\mathbb C,+,\times,\id)$, i.e., with the elements of $\mathbb{C}$ named by constants.}
\end{itemize}
Contrast this with the situation for differentially closed fields with a generic automorphism, where $\acfa$ appears as a proper reduct of the full structure induced on the field of constants.

In a similar vein, we prove that 
\begin{itemize}
\item[]
{\em the full structure induced by $\ccma$ on any simple nonalgebraic complex torus, or on any simply connected nonalgebraic strongly minimal compact complex manifold, is just the complex analytic structure together with $\sigma$.} (Theorem~\ref{purity})
\end{itemize}
In fact all of these results about induced structure follow from a general characterisation for pure stable embeddability in $TA$, for any complete stable theory $T$ for which $TA$ exists.
This characterisation, which is established in Proposition~\ref{intern-prop}, is that the sort in question should {\em internalise finite covers in $T$}, that is, relative to that sort in $T$ almost internality should imply outright internality.
That the projective line internalises finite covers in $\ccm$ is a (known) uniform version of the fact that a finite cover of a Moishezon space is again Moishezon.
For simple nonalgebraic complex tori and for simply connected nonalgebraic strongly minimal compact complex manifolds, the condition is proved in Section~\ref{finitecovers}.

But our primary interest is in the structure of the {\em finite-dimensional} types, that is, the types of tuples $c$ such that the complex dimension of the locus of $\big(c,\sigma(c),\dots,\sigma^n(c)\big)$ is uniformly bounded as $n$ grows.
For example, the types arising from meromorphic dynamical systems as described above are finite-dimensional.
The converse is almost true too: every finite-dimensional type in $\ccma$ comes from a pair $(V,\Gamma)$ where $V$ is a compact complex manifold and $\Gamma$ is a finite-to-finite meromorphic self-correspondence on~$V$.
See Section~\ref{findim-section} for details.
Our main theorem about finite-dimensional types is that they enjoy the canonical base property: 
\begin{itemize}
\item[]
{\em the canonical base of a finite-dimensional type in $\ccma$ is almost internal, over a realisation of the type, to the fixed field.} (Theorem~\ref{cbp})
\end{itemize}
Our proof follows the approach of Pillay and Ziegler~\cite{pillayziegler03} for $\acfa$; we define an appropriate notion of {\em jet space} in the context of $\ccma$.
As a consequence, we obtain that
\begin{itemize}
\item[]
{\em a finite-dimensional minimal type in $\ccma$ is either one-based or almost internal to the fixed field.} (Corollary~\ref{zd})
\end{itemize}
We expect, but do not prove, that the finite-dimensionality assumption in the above Zilber dichotomy statement can be removed.
We point out in Example~\ref{infDimTrivMin} that infinite-dimensional types of SU-rank one exist.

In a final section we explain how the methods of Hrushovski~\cite{maninmumford} and Chatzidakis-Hrushovski~\cite{acfa1} extend to our setting to describe the finite-dimensional minimal types that are nontrivial and one-based; roughly speaking they arise as generic types of quantifier-free definable subgroups of simple nonalgebraic complex tori.
See Proposition~\ref{mnt1b} for the precise statement.

We have included three appendices on material that we need and that is well known but that we either could not find good references for or that we judged was worth including to make the article as self-contained as possible.

\medskip
\subsection{Acknowledgements}
This collaboration began during the Spring 2014 MSRI programme {\em Model Theory, Arithmetic Geometry and Number Theory}.
The authors would like to thank MSRI for its hospitality and stimulating research environment.
We are also grateful to Fr\'ed\'eric Campana, Bradd Hart, and Matei Toma, for useful discussions.
We thank Jean-Beno\^{i}t Bost for pointing out, and suggesting a correction for,
an error in a previous version of the paper. Finally, we are grateful to the referee for a very thorough reading
of our paper, and for many useful suggestions.

\bigskip
\section{Generic automorphisms and induced structure}
\label{ta-section}

\noindent
In this section we review some of the basic theory of generic automorphisms (details of which can be found in~\cite{ChPi98}) and then prove a general theorem about when a generic automorphism induces as little structure on a definable set as possible.

Suppose $T=T^{\eq}$ is a complete (multi-sorted) stable theory admitting quantifier elimination in a language~$L$.
Consider the expanded language $L_\s:=L\cup\{\sigma\}$ where $\sigma$ is a unary function symbol, and in this language the universal theory
$$T_{\forall,\sigma}:=T_\forall\cup\{\sigma\text{ is an $L$-embedding}\}.$$
Recall that by $TA$ is meant the model companion of $T_{\forall,\sigma}$, when it exists.
In this section we will review what is known about $TA$ and give a characterisation for stable embeddedness in $TA$ of sorts from $L$.

Suppose $TA$ exists.
We work in a sufficiently saturated model $(\UU,\sigma)\models TA$.
So~$\mathbb U$ is a sufficiently saturated model of $T$.
As a general rule we will use notation such as $\tp$ and $\acl$ to refer to the $L$-structure~$\mathbb U$, and decorate with a ``$\sigma$", as in $\tpsigma$ and $\aclsigma$, when we wish to refer to $(\UU,\sigma)$.
The following properties of $TA$ are due to Chatzidakis and Pillay~\cite[$\S3$]{ChPi98}.
\begin{itemize}
\item[(1)]
For any set $A$, $\aclsigma(A)$ is the $\acl$-closure of the inversive closure of $A$.
Here the {\em inversive closure} of $A$ is by definition the set obtained from $A$ by closing off under $\sigma$ and $\sigma^{-1}$.
\item[(2)]\label{qr}
Quantifier reduction: $\tpsigma(a/A)=\tpsigma(b/A)$ if and only if there is an $L_\sigma$-isomorphism from $\aclsigma(Aa)$ to $\aclsigma(Ab)$ that takes $a$ to $b$ and fixes~$A$.
In particular, each completion
of $TA$ is determined by the isomorphism type of $(\acl(\emptyset),\sigma)$.
\item[(3)]\label{ss}
Every completion of $TA$ is simple and 
$$A\ind^{TA}_E B\text{ if and only if }\aclsigma(A)\ind^{T}_{\aclsigma(E)}\aclsigma(B).$$
If $T$ is superstable, then every completion of $TA$ is supersimple.
\end{itemize}
It is not in general true that $TA$ admits elimination of imaginaries, even though we assume that $T$ does.
Hrushovski has shown that for superstable $T$, a certain property of $T$ called {\em $3$-uniqueness} is equivalent to all completions of $TA$ eliminating imaginaries -- this is~\cite[Propositions~4.5 and~4.7]{HrGroupoids}.
However, as Bradd Hart pointed out to us, Hrushovski's arguments do give the following without assuming $3$-uniqueness:
\begin{itemize}
\item[(4)]\label{gei}
Geometric elimination of imaginaries: every element of $(\UU,\sigma)^{\eq}$ is interalgebraic with a finite tuple from~$\UU$. 
\end{itemize}

\begin{proof}
The proof of Proposition~4.7 of~\cite{HrGroupoids} shows that if $e=a/E$ is in $(\UU,\sigma)^{\eq}$, where $a$ is a finite tuple from the home sort and $E$ is a definable equivalence relation, then there exists in the $E$-class of $a$ an element, $a'$, such that $\displaystyle a'\ind^{TA}_{\aclsigma^{\eq}(e)\cap\UU}a$.
It follows that $e\in\aclsigma^{\eq}(a)\cap\aclsigma^{\eq}(a')\subseteq \aclsigma^{\eq}\big(\aclsigma^{\eq}(e)\cap\UU\big)$, as desired.
\end{proof}

\begin{remark}
Because of the potential failure of full elimination of imaginaries, we have to distinguish between $(\UU,\sigma)$, and $(\UU,\sigma)^{\eq}$.
As a general rule we will stay in the real sorts of $(\UU,\sigma)$. In the few cases we pass to $(\UU,\sigma)^{\eq}$ we will say so explicitly.
\end{remark}

Next we discuss the issue of what new structure $TA$ may induce on a given sort beyond that induced by $T$.
A notion relevant to this is stable embeddedness.
Let us recall the definition in general.

\begin{definition}
Working in a sufficiently saturated model of an arbitrary complete theory, a $0$-definable set $S$ is {\em stably embedded} if every definable subset of $S^n$ is definable using parameters from $S$, for all $n>0$.
\end{definition}

In a stable theory such as $T$, every definable set is stably embedded.
But we are interested in stable embeddedness in the (usually unstable) theory $TA$.

Let us make our question more precise.
Suppose $\PP$ is a set of sorts in $\UU$.
We want to know when $\PP$ is stably embedded in $(\UU,\sigma)$, and in that case what structure is induced.
Note that at the very least $(\UU,\sigma)$ induces on $\PP$ the automorphism $\sigma$, as well as all the structure that $\UU$ induces.
That is, if we let $L^{\PP}$ be the language consisting of a predicate symbol for each $L$-formula whose variables belong to sorts in $\PP$, then $(\PP, L^{\PP}_{\sigma})$ is a reduct of the full structure induced on $\PP$ by $(\UU,\sigma)$.
The question is, when is this the entire induced structure?
If this happens $\PP$ will be particularly tractable in $TA$.

\begin{remark}
\label{axiomatisetype}
By standard saturation arguments it is not hard to see that $\PP$ is stably embedded in $(\UU,\sigma)$ with induced structure given by $L^{\PP}_{\sigma}$ (maybe over some additional parameters) if and only if for every finite tuple $a$ from $\PP$, and any model $(M,\sigma)\elex(\UU,\sigma)$, $\tp_{L^{\PP}_{\sigma}}(a/\PP(M))\vdash\tpsigma(a/M)$.
\end{remark}

It turns out that understanding internality and almost internality in $T$ is key to answering this question. 
Given a complete type  $q(x)\in S(A)$, recall that $q(x)$ is {\em $\PP$-internal} if for some $B\supseteq A$, and some $a\models q$ with $\displaystyle a\ind^T_AB$, we have $a\in\dcl\big(B\PP\big)$.
If we replace $\dcl$ by $\acl$ in this definition we get the notion of {\em almost $\PP$-internal}.
The following condition is of central interest to us.

\begin{definition}
\label{finitecoversdef}
We say that $\PP$ {\em internalises finite covers in $T$} if almost $\PP$-internality implies $\PP$-internality.
\end{definition}

Recall that a {\em stationary type} is one that has a unique nonforking extension to any superset. Note that $\tp(a/A)$ is (almost) $\PP$-internal if and only if $\tp(a/\acl(A))$ is (almost) $\PP$-internal. Since 
types over algebraically closed sets in stable theories eliminating imaginaries are stationary, it is enough to consider stationary types in Definition \ref{finitecoversdef}.

\begin{example}
The projective line in $\ccm$ internalises finite covers, while the field of constants in $\dcf_0$ does not.
For the former see  Fact~3.1 of~\cite{moosapillay08}, and for the latter note that the generic type of $\delta(x^2-b)=0$ where $b$ is a fixed non-constant is almost internal but not internal to the field of constants (see~\cite[Lemma~3.1]{pillay06}).
\end{example}

The following characterisation will be useful in practice.

\begin{lemma}
\label{intern-rem}
The set of sorts $\PP$ internalises finite covers in $T$ if and only if for any $M\models T$ and any finite tuple $a\in\PP$,
\begin{equation}
\label{intern}
\acl(Ma)=\dcl\!\big(M,\acl(\PP(M)a)\cap\dcl(\PP)\big)
\end{equation}
\end{lemma}

\begin{proof}
If $\tp(b/A)$ is almost $\PP$-internal then there is a model $M\supseteq A$ independent from $b$ over $A$, and a finite tuple $a\in\PP$, such that $b\in\acl(Ma)$.
Then~(\ref{intern}) implies that $b\in\dcl(M\PP)$, showing that $\tp(b/A)$ is internal to $\PP$.

Conversely, assume that $\PP$ internalises finite covers and suppose $b\in\acl(Ma)$, where $a$ is a finite tuple from $\PP$. 
Then $\tp(b/M)$ is almost $\PP$-internal and hence $\PP$-internal.
As we are working over a model, and $T=T^{\eq}$, we get that $b$ is interdefinable  over $M$ with a finite tuple $c$ from $\dcl(\PP)$.
So $b\in\dcl(Mc)$.
On the other hand, $c\in\dcl(Mb)\subseteq\acl(Ma)$.
But as $a$ and $c$ are finite tuples from $\dcl(\PP)$, it follows by stability (namely, the fact that $\PP$ is stably embedded in $T$) that $c\in\acl(\PP(M)a)$.
Hence $b\in\dcl\!\big(M,\acl(\PP(M)a)\cap\dcl(\PP)\big)$, as desired.
\end{proof}

Now for the main result of this section which answers the question: when exactly is the full structure induced on $\PP$ by $(\mathbb U,\sigma)$ given by $(\PP,L^{\PP}_{\sigma})$?

\begin{proposition}\label{intern-prop}\mbox{}
\begin{itemize}
\item[(i)]
If $\PP$ internalises finite covers in $T$ then $\PP$ is stably embedded in $(\UU,\sigma)$, and the induced structure is that of $L^{\PP}_{\sigma}$, possibly with some parameters from $\PP(M)$ for any $(M,\sigma)\models T_\sigma:=T\cup\{\sigma\text{ is an $L$-automorphism}\}$.
\item[(ii)]
Conversely, assume that $\PP$ is stably embedded in $(\UU,\sigma)$ 
with induced structure given by~$L^{\PP}_{\sigma}$ (over parameters). 
Assume in addition that $\sigma$ restricted to $\acl(\emptyset)$ is the identity. Then $\PP$ internalises finite covers in $T$.
\end{itemize}
\end{proposition}

\begin{proof}
Since $T=T^{\eq}$, and $\dcl(\PP)$ internalises finite covers in $T$ if and only if $\PP$ does, we may also assume that $\PP=\PP^{\eq}$. 

Assume first that $\PP$ internalises finite covers in $T$. 
\begin{claim}\label{Cl:Tsigma}
Let $(M,\sigma)\models T_\sigma$ be contained in $(\UU,\sigma)$, and let $a_1,a_2$ be finite tuples from $\PP$ such that $\tp_{L^{\PP}_{\sigma}}(a_1/\PP(M))=\tp_{L^{\PP}_{\sigma}}(a_2/\PP(M))$. Then 
$\tpsigma(a_1/M)=\tpsigma(a_2/M)$. 
\end{claim}
\begin{proof}[Proof of Claim \ref{Cl:Tsigma}]
By assumption, there is an $L^{\PP}_\sigma$-automorphism of $\PP$  that fixes $\PP(M)$ pointwise and takes $a_1$ to $a_2$.
For $i=1,2$, set
$$A_i:=\aclsigma\big(\PP(M)a_i\big)\cap\PP=\acl\big(\PP(M)(\sigma^{n}a_i)_{n\in\mathbb Z}\big)\cap\PP$$
We obtain an $L^{\PP}_\sigma$-isomorphism $f:A_1\to A_2$ that fixes $\PP(M)$ pointwise and takes $a_1$ to $a_2$.
By stable embeddedness of $\PP$ in $\UU$, the $L^{\PP}$-isomorphism $f$ extends to an $L$-isomorphism
$$F:\dcl(MA_1)\to\dcl(M A_2)$$
fixing $M$ pointwise.
Note that $M\cup A_i$, and hence $\dcl(MA_i)$, are $\sigma$-invariant.
Moreover, $F$ is in fact an $L_{\sigma}$-isomorphism; this is because $F$ and $\sigma$ commute on $M$ since $F|_M=\id$ and they commute on $A_1$ as $F|_{A_1}=f$ which is an $L^{\PP}_\sigma$-isomorphism.
On the other hand, by Lemma~\ref{intern-rem},
$$\dcl(MA_i)=\acl\big(M(\sigma^{n}a_i)_{n\in\mathbb Z}\big)=\aclsigma(Ma_i)$$
since $\PP$ internalises finite covers in $T$.
(Note that Lemma~\ref{intern-rem} applies to infinite tuples by the finite character of $\acl$ and $\dcl$.)
So by quantifier reduction in $TA$, namely~(2) of page~\pageref{qr}, $F$ witnesses $\tpsigma(a_1/M)=\tpsigma(a_2/M)$, as desired.
\end{proof}

By Remark~\ref{axiomatisetype}, Claim~\ref{Cl:Tsigma} applied to models of $TA$ shows stable embeddedness of $\PP$ in $(\UU,\sigma)$ with induced structure given by $L^{\PP}_{\sigma}$ (over parameters). Statement (i) follows, using Claim~\ref{Cl:Tsigma} for arbitrary models of $T_\sigma$.

\medskip

To prove (ii), we now assume that $\sigma$ restricted to $\acl(\emptyset)$ is the identity, and that 
$\PP$ is stably embedded in $(\UU,\sigma)$ 
with induced structure given by~$L^{\PP}_{\sigma}$, say over parameters from $\PP(M)$, where $(M,\sigma)\elex(\UU,\sigma)$.
Suppose for contradiction that $\PP$ does not internalise finite covers in $T$. By the characterisation given in Lemma \ref{intern-rem}, there 
is $M_0\models T$ and a finite tuple $b\in\PP$ such that 
\begin{equation}\label{eq:non-intern}
\acl(M_0b)\supsetneq\dcl\big(M_0,\acl(\PP(M_0)b)\cap\PP\big).
\end{equation}

By universality of $(\UU,\sigma)$ and as $\sigma\upharpoonright_{\acl(\emptyset)}=\id$, we may assume that $(M_0,\id)\subseteq(\UU,\sigma)$. Increasing $M$ and conjugating $b$ if necessary, we may assume in addition 
that $M_0\subseteq M$ and $b\ind_{M_0} M$.

Note that it follows from stable embeddedness of $\PP$ in $T$ that for any $N\models T$ we have $\acl(\PP(N)b)\cap\PP=\acl(Nb)\cap\PP$. We will freely use this equality in what follows and denote the set 
$\acl(Nb)\cap\PP$ by $\acl^{\PP}(Nb)$.

\begin{claim}\label{Cl:Mult}
Let $c\in\acl(M_0\cup\acl^{\PP}(M_0b))$. Then $$\tp(c/M_0\cup\acl^{\PP}(M_0b))\vdash\tp(c/M\cup\acl^{\PP}(Mb)).$$
\end{claim}
\begin{proof}[Proof of Claim \ref{Cl:Mult}]
Let $l$ be the number of conjugates of $c$ over $M\cup\acl^{\PP}(Mb)$. There are $L$-formulas $\psi(z,x,w)$ and $\chi(y,z,w)$, and a tuple $m\in M$ such that 
\begin{itemize}
\item[(I)] whenever $\psi(d',b',m')$ holds, we have $d'\in\acl(b'm')$; 
\item[(II)] $T\models\forall z,w\,\exists^{\leq l}y\,\chi(y,z,w)$, and
\item[(III)] $\tp(bc/M)\vdash\exists z\in\PP[\psi(z,x,m)\wedge\chi(y,z,m)]$.
\end{itemize}
Since $bc\ind_{M_0}M$, by (III) there is $m_0\in M_0$ such that 
$$\tp(bc/M_0)\vdash\exists z\in\PP[\psi(z,x,m_0)\wedge\chi(y,z,m_0)]$$
Choose $d_0\in\PP$ such that $\models\psi(d_0,b,m_0)\wedge\chi(c,d_0,m_0)$. Thus, 
$d_0\in\acl^{\PP}(M_0b)$ by (I). Moreover, as $\models\chi(c,d_0,m_0)$, we infer from (II) that $\tp(c/M\cup\acl^{\PP}(Mb))$ is isolated by $\chi(y,d_0,m_0)$, a 
formula over $M_0\cup\acl^{\PP}(M_0b)$. This proves the claim.
\end{proof}

We now claim that there exist finite tuples $b_1,b_2\in\PP$ such that

\begin{itemize}
\item[(1)]
$\sigma\upharpoonright_{\acl^{\PP}(M_0b_i)}=\id$ for $i=1,2$;
\item[(2)] there is an $L_\sigma$-automorphism $F:(\acl^{\PP}(Mb_1),\sigma)\cong(\acl^{\PP}(Mb_2),\sigma)$ satisfying $F(b_1)=b_2$ and $F\upharpoonright_{\PP(M)}=\id$, and
\item[(3)]
$\sigma\upharpoonright_{\acl(M_0b_1)}=\id$ but $\sigma\upharpoonright_{\acl(M_0b_2)}\neq\id$.
\end{itemize}


To prove this, first note that the permutation $\tau'$ of the set $M\cup\acl^{\PP}(M_0b)$ given by $\tau'\upharpoonright_M=\sigma$ and $\tau'\upharpoonright_{\acl^{\PP}(M_0b)}=\id$ is elementary, as 
$\sigma\upharpoonright_{M_0}=\id$ and $\acl^{\PP}(M_0b)\ind_{M_0}M$.
Next, choose an arbitrary elementary permutation $\tau$ of $M\cup\acl^{\PP}(Mb)$ extending~$\tau'$. Let $\rho\in\gal(M_0\cup\acl^{\PP}(M_0b))$, i.e., $\rho$ is 
an elementary permutation of $\acl(M_0b)$ fixing pointwise the set $M_0\cup\acl^{\PP}(M_0b)$.
It then follows from Claim \ref{Cl:Mult} that the map $\tau\cup\rho$ is an elementary permutation of $B:=M\cup\acl^{\PP}(Mb)\cup\acl(M_0b)$. 

For $\rho_1=\id\in\gal(M_0\cup\acl^{\PP}(M_0b))$, let $f_1:(B,\tau\cup\rho_1)\hookrightarrow (\UU,\sigma)$ be an $M$-embedding and put $b_1=f_1(b)$. Now let $\rho_2\in\gal(M_0\cup\acl^{\PP}(M_0b))$ be 
different from the identity. Such a map exists by \eqref{eq:non-intern}. We may choose an 
$M$-embedding $f_2:(B,\tau\cup\rho_2)\hookrightarrow (\UU,\sigma)$ and put $b_2=f_2(b)$. Properties~(1)--(3) hold by construction.

\smallskip

Next, we claim that $\tp_{L^{\PP}_{\sigma}}(b_1/\PP(M))=\tp_{L^{\PP}_{\sigma}}(b_2/\PP(M))$.
By~(2) this follows from quantifier reduction (property (2) on page \pageref{qr}) in $T^{\PP}A$ where $T^{\PP}=\operatorname{Th}_{L^{\PP}}(\PP)$, once we know that $T^{\PP}A$ exists and that $(\PP,\sigma)$ is a model of it.
The latter facts are widely known, but as we could not find a reference we prove it as Fact~\ref{Fact:TPA} of the appendix.

On the other hand, by~(3) we clearly have $\tpsigma(b_1/M_0)\neq\tpsigma(b_2/M_0)$, and so in particular $\tpsigma(b_1/M)\neq\tpsigma(b_2/M)$.
This contradicts the assumption that $\PP$ is stably embedded in $(\UU,\sigma)$ with induced structure given by~$L^{\PP}_{\sigma}$ (over parameters from $\PP(M)$).
\end{proof}

We expect that the following question has a positive answer.

\begin{question}
Does Proposition \ref{intern-prop}(ii) hold without the assumption on $\sigma\upharpoonright_{\acl(\emptyset)}$?
\end{question}

In the special case where $\dcl(\emptyset)$ is a model of $T$ (which is the case for the theory $T=\ccm$ that we consider in the next section), we may infer from Proposition \ref{intern-prop} the following clean characterisation.

\begin{corollary}\label{Cor:dcl-model-stab-emb}
Suppose that $\dcl(\emptyset)\models T$. Then the following are equivalent:
\begin{itemize}
\item $\PP$ internalises finite covers in $T$;
\item $\PP$ is stably embedded in $TA$, with induced structure given by~$L^{\PP}_{\sigma}$ over parameters;
\item $\PP$ is stably embedded in $TA$, with induced structure given by~$L^{\PP}_{\sigma}$ over the empty set. 
\end{itemize}
\end{corollary}

\bigskip
\section{$\ccma$ and the failure of elimination of imaginaries}
\label{not3unique}

\noindent
We now specialise to the case when $T=\ccm$.
Let us first review some of our conventions and notations regarding this theory.

By a {\em complex variety} we will mean a reduced and irreducible complex analytic space.
We say that a property holds {\em generally} if it holds outside a proper complex analytic subset.
Recall that $\A$ is the structure where the sorts are the compact complex varieties and the basic relations are the complex analytic subsets of the cartesian products of the sorts.
The first-order theory of $\A$ is denoted by $\ccm$.
One of the sorts of $\A$ is the projective line $\mathbb P(\CC)$.
The complex field $(\CC,+,\times)$ is definable in this sort, and in fact the full induced structure on $\mathbb P(\CC)$, which by Chow's Theorem is just that of the algebraic sets over $\CC$, is bi-interpretable with the complex field.

We will usually work in a sufficiently saturated elementary extension $\A'$ of $\A$.
The corresponding elementary extension of the complex field, i.e. the interpretation in $\A'$ of the complex field, will be denoted by $(K,+,\times)$.
It is up to definable isomorphism the only infinite field definable in $\A'$, see~\cite[Corollary~4.8]{ret}.

The complex analytic sets induce a noetherian topology on the sorts of $\A'$ too: if $X$ is a compact complex variety then by a {\em closed} subset of $X(\A')$ is meant a set of the form $F:=Y(\A')_s$ where $Y$ is a complex analytic subset of $S\times X$ for some compact complex variety $S$, and $s\in S(\A')$.
If $s$ comes from a given set of parameters $B$, then we also say that $F$ is {\em $B$-closed}.
A $B$-closed set is {\em $B$-irreducible} if it cannot be written as the union of two proper $B$-closed sets.
If $F$ is a $B$-irreducible $B$-closed set, a {\em generic point of $F$ over $B$} is an element $c\in F$ that is not contained in any proper $B$-closed subset. In this case we also say that $F$ is the {\em locus of $c$ over $B$}, denoted by $F=\loc(c/B)$.
We say $F$ is {\em irreducible} if it is absolutely irreducible in the sense that it cannot be written as a union of two proper closed subsets (over any parameters).
Even though the closed subsets of $X(\A')$ are not themselves complex analytic spaces, it still makes sense to talk about their ``complex" {\em dimension} since in the standard model the complex dimension of a complex analytic set is definable in parameters.
For proofs of the claims made here, and for more details on this ``nonstandard Zariski" topology, see~\cite[$\S$2]{ret}.

\medskip

In the introduction we gave motivation for studying $\ccm_{\forall,\sigma}$, and hence its model companion $\ccma$, coming from meromorphic dynamics.
The first two authors and Gavrilovich showed in~\cite{bgh} that $\ccma$ exists.
(In fact, they showed the existence of $TA$ for $T$ the theory of any $\omega_1$-compact noetherian topological structure with quantifier elimination and in which  irreducibility is definable.)
So $\ccma$ satisfies properties~(1) through~(4) discussed at the beginning of Section~\ref{ta-section}.
Moreover, as $\acl(\emptyset)=\dcl(\emptyset)=\mathcal A$, quantifier reduction gives us that $\ccma$ is complete.

The goal of this section is to prove that while $\ccma$ admits geometric elimination of imaginaries, it does not eliminate imaginaries outright.
As we have mentioned before, by work of Hrushovski this reduces to proving the following theorem about $\ccm$.

\begin{theorem}\label{thm:non3unique}
$\ccm$ does not satisfy $3$-uniqueness in the sense of~\cite[$\S4$]{HrGroupoids}. 
That is, working in a sufficiently saturated $\mathcal A'\models\ccm$,
there exist elements $b,x,y,z$ with $x,y,z$ independent over $b$ and such
    that 
    \[ \operatorname{acl}(bxy) \cap \operatorname{dcl}\big(\operatorname{acl}(bxz),\operatorname{acl}(byz)\big) \neq \operatorname{dcl}\big(\operatorname{acl}(bx),\operatorname{acl}(by)\big) .\]
\end{theorem}

Our witness to the failure of $3$-uniqueness will use some of the theory of holomorphic line bundles on complex manifolds.
We suggest~\cite[$\S1.1$]{griffithsharris} for more details and as a general reference for this material.

In fact, it is convenient for us to work instead with the corresponding $\CC^*$-bundles.
Recall that a {\em holomorphic $\CC^*$-bundle} on a complex manifold $M$ is a complex manifold $P$ with a holomorphic surjective map $P\to M$ for which there exists a (euclidean) open cover $\{U_i\}$ of $M$ such that $P|_{U_i}$ is biholomorphic to $U_i\times \mathbb C^*$ over $U_i$, and such that the corresponding biholomorphisms $U_i\cap U_j\times\mathbb C^*\to U_i\cap U_j\times\mathbb C^*$ are of the form $(x,p)\mapsto(x,g_{ij}(x)p)$ for some (holomorphic) $g_{ij}:U_i\cap U_j\to\mathbb C^*$.
One obtains a uniform holomorphic group action of $\CC^*$ on the fibres of $P\to M$.
Since the action by multiplication of $\CC^*$ on $\CC^*$ extends uniquely and holomorphically to $\CC$, an alternative description of $\CC^*$-bundles is that they are holomorphic line bundles with their zero sections removed.
Moreover, as that action extends also uniquely and holomorphically to $\mathbb P(\CC)$, we can embed $P$ as a Zariski open subset of a projective line bundle over $M$, which we will denote by $ P^{\operatorname{cl}}\to M$.
Hence, if $M$ is compact then $P\to M$, as well as the action of $\CC^*$ on the fibres, is definable in $\ccm$.
The set of holomorphic $\CC^*$-bundles over $M$, up to a natural notion of isomorphism, can be identified with the cohomology group $H^1(M,\mathcal O_M^*)$, where $\mathcal O_M^*$ is the sheaf of nonvanishing holomorphic functions on $M$.
Under this identification the group structure induced on the set of holomorphic $\CC^*$-bundles is given by pointwise multiplication of the $g_{ij}$'s with respect to a common choice of $\{U_i\}$.  
We will denote the element of $H^1(M,\mathcal O_M^*)$ corresponding to $P$ by $[P]$.

The following fact was pointed out to us by Matei Toma.

\begin{fact}
\label{fact:K3}
There exists a strongly minimal simply connected smooth compact complex surface $X$ with $H^1(X,\mathcal O_X^*) \isom  \mathbb Z$.
\end{fact}
\begin{proof}
The examples come from $K3$ surfaces, which are certain simply connected smooth compact complex surfaces, details about which can be found in~\cite[Chapter~VIII]{bpv}.
In~\cite[Theorem 1]{LePotier} a family of $K3$ surfaces with $H^1(X,\mathcal O_X^*) \isom  \mathbb Z$ are constructed; we need only verify that among these there are some that are strongly minimal.
This is in fact the case: by choosing $g<0$ in that theorem one obtains $X$ with the additional property that every irreducible curve on $X$ has self-intersection $\leq -4$.
 But it is a general fact that on $K3$ surfaces curves always have self-intersection at least $-2$, see~\cite[VIII.3.7(ii)]{bpv}.
 It follows that such $X$ can have no curves and are therefore strongly minimal.
\end{proof}

Fix $X$ as given by Fact~\ref{fact:K3}.

The following lemma shows that passing to an open subset $X' \subseteq X$
does not cause cohomology classes of $\CC^*$-bundles on $X$ to become more divisible.
We are grateful to Jean-Beno\^{i}t Bost for pointing out an error in a previous formulation of this lemma,
and for indicating the correct statement and proof given here.

\begin{lemma}\label{lem:openH1}
    For any nonempty Zariski open subset $X' \subseteq  X$,
    the map
    \[ H^1(X,\mathcal O_X^*) \rightarrow  H^1(X',\mathcal O_{X'}^*) \]
    induced by the inclusion is injective and has torsion-free cokernel.
\end{lemma}
\begin{proof}
The only property of $X$ that is used in this proof is that it is a strongly minimal compact complex manifold of dimension greater than $1$.
On the other hand we will be making use of more sophisticated tools from homological algebra.

Let $Z := X\setminus X'$, which is finite by strong minimality.
By~\cite[Corollaire~I.2.9]{SGA2}, our homomorphism $H^1(X,\mathcal O_X^*) \rightarrow  H^1(X',\mathcal O_{X'}^*)$ fits into a long exact sequence
\[ H^1_Z(X,\O_X^*) \rightarrow  H^1(X,\O_X^*) \rightarrow  H^1(X,\O_{X'}^*) \rightarrow  H^2_Z(X,\O_X^*), \]
where the local cohomology $H^i_Z(X,\underline{\ \ })$ is the $i$th derived functor of the functor which takes a sheaf of abelian groups $\mathcal F$ on $X$ to the group of global section of~$\mathcal F$ supported on $Z$ (see \cite{SGA2}).
It therefore suffices to prove
\begin{itemize}
\item[($*$)]
$H^1_Z(X,\O_X^*)=0$ and $H^2_Z(X,\O_X^*)$ is torsion-free.
\end{itemize}
Let $V \subseteq  X$ be a union of disjoint open polydiscs centred at the points of $Z$.
By \cite[Proposition~I.2.2]{SGA2}, for all $i$,
\[ H^i_Z(X,\O_X^*) \cong  H^i_Z(V,\O_V^*) .\]
We are thus reduced to proving ($*$) for $V$ rather than $X$.

Now consider the exponential exact sequence
\[ 0 \rightarrow  \ZZ \rightarrow \O_{V} \rightarrow  \O_{V}^*\rightarrow 1.\]
From this we obtain the long exact sequence in cohomology
\[ \cdots \rightarrow H^i(V,\ZZ) \rightarrow  H^i(V,\O_{V}) \rightarrow H^i(V,\O_{V}^*) \rightarrow  H^{i+1}(V,\ZZ) \rightarrow\cdots .\]
Since $V$ is a disjoint union of polydiscs, $H^j(V,\ZZ)=0$ and $H^{j}(V,\O_{V})=0$ for all $j>0$.
See for example~\cite[Corollary 10.1.5]{taylor} for a proof of the latter.
Hence $H^i(V,\O_{V}^*)=0$ for all $i>0$.

Now consider $V' := V \setminus Z$.
As $V$ is a finite disjoint union of contractible 4-dimensional real
manifolds and $Z$ is finite, we still have $H^i(V',\ZZ)=0$ for $i=1,2$.
Hence the exponential exact sequence for $V'$ yields
$H^1(V',\O_{V'}) \cong  H^1(V',\O_{V'}^*)$.

Now consider the long exact sequence in local cohomology \cite[Corollaire~I.2.9]{SGA2}:
\begin{align*}
  H^0(V,\O_V^*) \rightarrow  & H^0(V',\O_{V'}^*) \rightarrow  H^1_Z(V,\O_V^*) \rightarrow  H^1(V,\O_V^*) \rightarrow \\
    & H^1(V',\O_{V'}^*) \rightarrow  H^2_Z(V,\O_V^*) \rightarrow  H^2(V,\O_V^*) \rightarrow  \ldots
\end{align*}
By Hartogs' extension theorem~\cite[Theorem 2.3.2]{hartogs}, the first map is an isomorphism.
Since also $H^1(V,\O_V^*) = 0$, we conclude
$H^1_Z(V,\O_V^*) = 0$, as desired.
For the second part of ($*$), since $H^2(V,\O_V^*) = 0$,
we have $H^2_Z(V,\O_V^*) \cong  H^1(V',\O_{V'}^*) \cong  H^1(V',\O_{V'})$, which is a complex vector space, so $H^2_Z(V,\O_V^*)$
is torsion-free.
\end{proof}

\begin{lemma}\label{lem:Bacldcl}
    If $b\in X(\A')$ is generic then $\operatorname{acl}(b) = \operatorname{dcl}(b)$
\end{lemma}
\begin{proof}
This only uses the fact that $X$ is simply connected, smooth, and has no subvarieties of codimension 1.
    Let $c \in \operatorname{acl}(b)$, $Y:=\loc(c,b)$, and $\pi:Y \to  X$ the  generically finite surjective holomorphic map given by the second co-ordinate projection.
    Consider the {\em Stein factorisation} of $\pi$
    $$\xymatrix{
    Y\ar[rr]^f\ar[dr]_{\pi}&& \widetilde Y\ar[dl]^{\widetilde\pi}\\
    & X}$$
    namely, $\widetilde\pi$ is everywhere finite-to-one and the fibres of $f$ are connected.
    So in this case $f$ is a bimeromorphism.
    Next, precomposing with a {\em normalisation}, namely a finite bimeromorphic map $\widehat Y\to\widetilde Y$ with $\widehat Y$ a normal compact complex variety, we obtain a (still finite) surjective map $\widehat\pi:\widehat Y\to X$.
    Because $\widehat Y$ is normal and $X$ is smooth we can apply the purity of branch locus theorem (see, for example, Theorem~VII.1.6 of~\cite{scv7}) which tells us that the ramification locus of $\widehat\pi$, namely the set of points in $\widehat Y$ where $\widehat\pi$ is not locally biholomorphic, say $E$, is a complex analytic set that is either empty or of codimension $1$.
    In the latter case, as $\widehat\pi$ is finite, it would follow that $\widehat\pi(E)$ is a curve in $X$, contradicting strong minimality of this surface.
    Hence $\widehat\pi$ is unramified.
    But simple-connectedness of $X$ then forces $\widehat\pi$, and hence $\widetilde\pi:\widetilde Y\to X$, to be an isomorphism.
    It follows that the original $\pi:Y\to X$ is bimeromorphic, and so $c\in\dcl(b)$, as desired.
    \end{proof}

Fix a nontrivial holomorphic $\CC^*$-bundle $P\to X$ such that $[P]$ is a generator of $H^1(X,\mathcal O_X^*)$, given to us by Fact~\ref{fact:K3}.
Let $\pi: P^{\operatorname{cl}}\to X$ be the ambient projective line bundle in which $P$ lives as  a Zariski open set.

\begin{proposition}\label{prop:Xacldcl}
    If $a \in  P^{\operatorname{cl}}(\A')$ is generic then $\operatorname{acl}(a) = \operatorname{dcl}(a)$.
\end{proposition}

\begin{proof}
We are grateful to Will Sawin for his answer to a mathoverflow question~\cite{sawin} which pointed us toward analysing the situation in terms of \'etale covers.

Let $z\in\acl(a)$ and set $Z:=\loc(z)$.
Arguing exactly as in the beginning of the proof of Lemma~\ref{lem:Bacldcl}, after replacing $z$ with some $z'$ such that $z'\in\operatorname{acl}(a)$ and $a,z\in \operatorname{dcl}(z')$, 
we may assume that there is a finite surjective holomorphic map $f:Z\to P^{\operatorname{cl}}$ whose ramification locus $E$ is either empty or of pure codimension one.

We first argue that $f(E)\subseteq P^{\operatorname{cl}}\setminus P$.
Indeed, if $E\neq\emptyset$ then as $f$ is finite, $f(E)$ is of pure codimension one in $ P^{\operatorname{cl}}$, and hence is of dimension $2$.
Let $S$ be an irreducible component of $f(E)$.
Since the fibres of $\pi$ are one-dimensional, it follows that $\pi(S)$ is infinite in $X$, and so by strong minimality is all of $X$.
Thus $S$ projects generically finitely onto $X$, and so by Lemma~\ref{lem:Bacldcl}, the projection is in fact generically one-to-one.
That is, $S$ is the image of a meromorphic section to $ P^{\operatorname{cl}}\to X$.
So if $S\cap P\neq\emptyset$, i.e., if $S$ is generically contained in $P$, then $P\to X$ has a meromorphic section $s:X\to P$.
It follows that $P\to X$ is bimeromorphically trivial: there is a nonempty Zariski open subset $X'\subseteq X$ such that $P|_{X'}$ is isomorphic to $X'\times\mathbb \CC^*$ over $X'$.
Indeed, the bimeromorphism will send $p\in P|_{X'}$ to $\big(\pi(p),z\big)$ where $z$ is the nonzero complex number such that $z\cdot p=s\big(\pi(p)\big)$.
Hence $[P|_{X'}]=0$ in $H^1(X',\mathcal O_{X'}^*)$.
By~\ref{lem:openH1}, this contradicts the nontriviality of $[P]$.
So $f(E)\subseteq P^{\operatorname{cl}}\setminus P$.

Hence, setting $Q:=f^{-1}(P)$, $f|_Q:Q\to P$ is an unramified finite cover.
We aim to prove that, after possibly removing finitely many points from $X$, $Q$ is also a $\CC^*$-bundle and that in the local trivialisations $f|_Q$ becomes $\id_U\times[n]:U\times\CC^*\to U\times\CC^*$ for some $n>0$, where $[n]:\CC^*\to\CC^*$ is the raising to the power $n$ map.

To that end, note first of all that as $\operatorname{CCM}$ eliminates imaginaries, Lemma~\ref{lem:Bacldcl} applied to $b:=\pi(a)\in X(\A')$ implies that $\tp(z/b)$ is stationary.
Hence the general fibres of $Z\to X$ are irreducible (see Lemmas~2.7 and~2.11 of~\cite{ret}).
It follows that the general fibres of $Q\to X$ are irreducible, so in particular connected for the Zariski topology and thus for the euclidean topology, too.
Let $X'\subseteq X$ be a nonempty Zariski open set over which this happens, and let $U\subseteq X'$ be an open ball such that $P|_U$ is isomorphic to $U\times\CC^*$ over~$U$.
Then $Q|_U$ is connected 
and we have an induced holomorphic unramified finite covering map $\alpha: Q|_U\to U\times\CC^*$ over~$U$.
Since $U\times \CC^*$ has fundamental group $\mathbb Z$, any topological covering map is isomorphic as a topological covering map to some $\id_U \times [n]: U\times \CC^* \maps  U\times \CC^*$.
Now $\alpha$ is in particular a topological
	covering map, so there is a topological covering map isomorphism $\theta :  Q|_{U}\mapsto  U\times \CC^*$ such that $\alpha = (\id_U\times[n])\circ\theta$.
	But $\alpha$ and $[n]$ are locally
	biholomorphic, so it follows that $\theta$ is locally biholomorphic,
	and so is an isomorphism of complex manifolds.
	That is, $Q|_U\to P|_U$ is isomorphic to $\id_U \times [n]: U\times \CC^* \maps  U\times \CC^*$.
	
	What we have shown is that the $\CC^*$-bundle structure on $P|_{X'}$ lifts to one on $Q|_{X'}$ and that in $H^1(X',\mathcal O_{X'}^*)$, $n[Q|_{X'}]=[P|_{X'}]$.
	It follows by Lemma~\ref{lem:openH1} that $[P]$ is $n$-divisible in $H^1(X,\mathcal O_X^*)$.
This contradicts the fact that $[P]$ generates $H^1(X,\mathcal O_X^*)$, unless $n=1$.
That is, $f$ restricts to an isomorphism $Q|_{X'}\to P|_{X'}$, and as $Q|_{X'}$ is Zariski open in $Z$ and $P|_{X'}$ is Zariski open in $P^{\operatorname{cl}}$, we have that $f:Z\to P^{\operatorname{cl}}$ is bimeromorphic.
So $z\in\dcl(a)$.
\end{proof}

\begin{proof}[Proof of Theorem~\ref{thm:non3unique}]
We have a uniformly definable action of $\CC^*$ on the fibres of $P\to X$.
So if $b\in X(\A')$ is generic, $P_b$ is a definable (principal) homogenous space for $K^*$.
Let $a\in P_b$ be generic over $b$.
Note that by the irreducibility of $P$ and the additivity of dimension, $a$ is generic in $P$.
Let $\phi,\theta\in K^*$ be independent generics over $a$.
    So $(a,\phi a,\theta a)\in P_b^3$ is an independent triple over $b$, and this will be our witness to non-3-uniqueness.
    Fix $n>1$ and choose $n$th roots $\phi'$ and $\theta'$, of $\phi$ and $\theta$ respectively.
     Then $\frac{\phi'}{\theta'}\in\acl(\frac{\phi}{\theta})\subseteq \acl(\phi a,\theta a)$.
    On the other hand, $\theta'\in\acl(\theta)\subseteq \acl(a, \theta a)$.
    It follows that
    $\phi'\in\dcl(\frac{\phi'}{\theta'},\theta')\subseteq\dcl\big(\acl(\phi a,\theta a), \acl(a, \theta a)\big)$.
    So
    $$\phi'\in\acl(\phi a,a)\cap\dcl\big(\acl(\phi a,\theta a), \acl(a, \theta a)\big)$$
    Suppose toward a contradiction that $\phi'\in\dcl\big(\acl(\phi a),\acl(a)\big)$.
  Since both $a$ and $\phi a$ are generic in $P^{\operatorname{cl}}(\A')$, by Proposition~\ref{prop:Xacldcl}, it follows that $\phi'\in\dcl(\phi a,a)=\dcl(a,\phi)$.
  But note that $\tp(a/\phi)$ is stationary (since $\tp(a)$ is and  $a\ind\phi$) and so $\dcl(a,\phi)\cap\acl(\phi)=\dcl(\phi)$.
  Hence $\phi'\in\dcl(\phi)$, which contradicts $n>1$.
\end{proof}

\begin{corollary}
\label{notEI}
$\ccma$ does not eliminate imaginaries.
\end{corollary}

\begin{proof}
Hrushovski has shown that for superstable $T$, elimination of imaginaries for (all completions of) $TA$ is equivalent to $T$ satisfying $3$-uniqueness -- this is~\cite[Propositions~4.5 and~4.7]{HrGroupoids}.
\end{proof}

\begin{remark}
\label{notEIExplicit}
  We can explicitly describe a finite imaginary which is not eliminated.
  Let $\pi: P\to X$ be as above.
    Let $E$ be the relation on the fixed set
    $P^f := \{ a \in P \;|\; \sigma(a)=a \}$
    of $P$ defined by
      \[ a E a' \iff \exists\theta \in K^* \,(a' = \theta^2 a \wedge \sigma(\theta)=\theta) .\]
  This is a $0$-definable equivalence relation, and as the subgroup of squares of the multiplicative group of the fixed field has index two, there are two equivalence
  classes on each fibre $P_b^f$.
  Let $a \in P^f$ be generic in $P$; such exists by the axioms for CCMA.
  Let $b := \pi(a)$.
  Let $\theta\in K^*$ be generic over $a$
    and such that $\sigma(\theta) = -\theta$.
  By Proposition~\ref{prop:Xacldcl} and Lemma~\ref{lem:Bacldcl},
    $\aclsigma(a) = \acl(a) = \dcl(a)$ and $\aclsigma(b) = \acl(b) = \dcl(b)$.
  Now $\theta^2 a\in P_b^f$ is also generic in $P_b$,
    and so it follows from the quantifier reduction that
    $\tpsigma(a / \aclsigma(b)) = \tpsigma(\theta^2 a / \aclsigma(b))$.
  But $P_b^f / E = \{ a/E, (\theta^2 a) / E \}$.
  So $a/E \in \aclsigma^{\eq}(b) \setminus \dcl_{\sigma}^{\eq}(\aclsigma(b))$.
\end{remark}

\bigskip
\section{Induced structure in $\ccma$}
\label{finitecovers}

\noindent
Fix a sufficiently saturated model $(\mathcal A',\sigma)\models\ccma$. In what follows, we will apply the results of Section~\ref{ta-section} to the theory $\ccma$. Since $\dcl(\emptyset)=\A\models\ccm$, 
the characterisation of stable embeddedness given in Corollary \ref{Cor:dcl-model-stab-emb} holds in $\ccma$, but we will only make use of one direction, namely the criterion stated in Proposition~\ref{intern-prop}(i). 

We obtain directly an understanding of what structure $\ccma$ induces on projective algebraic varieties.
Recall that, for a compact complex variety $X$, by $L^X$ we mean the language of analytic sets restricted to $X$, that is the language where there is a predicate for each complex-analytic subset of each cartesian power of $X$.

\begin{theorem}
\label{purityalg}
If $X$ is a complex projective variety then $X(\mathcal A')$ is stably embedded in $(\mathcal A',\sigma)$ and the induced structure is $\big(X(\mathcal A'),L^X_{\sigma}\big)$.
\end{theorem}

\begin{proof}
This follows immediately from Proposition~\ref{intern-prop} by the fact that complex projective varieties internalise finite covers in $\ccm$; see Fact~3.1 of~\cite{moosapillay08}.
\end{proof}

Theorem~\ref{purityalg} tells us that $\acfa$ (with the elements of $\CC$ named) is ``purely stably embedded" in $\ccma$.
Indeed, note that $(K,+,\times)$, the interpretation of the complex field  in $\A'$, with the elements of $\CC$ named, is $0$-definably bi-interpretable with $(X(\A'), L^{X})$ where $X=\mathbb P(\CC)$, and hence by applying the theorem  we have that the structure induced on $K$ by $(\A',\sigma)$ is precisely $(K,+,\times,\sigma)\models\acfa$.

\begin{remark}[{Bustamante \cite[Prop. 2.3.5]{bustamante}}]
$\acfa$ is also interpretable in the theory of differentially closed fields (of characteristic 0) with a generic automorphism, by considering the field of constants.
However, $\acfa$ is not in this way purely stably embedded in $\dcfa$.
Indeed, let $(K,+,\times,\delta,\sigma)\models\dcfa$ with $C=\operatorname{const}(\delta)$, and $b\in\fix(\sigma)\setminus C$. Then 
$\{y\in C\,\mid\, \exists x\,[\sigma(x)=x\wedge x^2-b=y]\}$ is not definable in~$(C,+,\times,\sigma)$. 
\end{remark}

What about the structure induced on other compact complex varieties?
In the rest of this section we aim to prove the following theorem that deals with some other compact complex manifolds of particular interest.

\begin{theorem}
\label{purity}
Suppose $X$ is either a nonalgebraic simple complex torus or a strongly minimal simply connected nonalgebraic compact complex manifold.
Then $X(\mathcal A')$ is stably embedded in $(\mathcal A',\sigma)$ with induced structure $(X(\mathcal A'),L^X_\sigma)$.
\end{theorem}

This theorem will follow also from Proposition~\ref{intern-prop}, once we know that the relevant compact complex manifolds internalise finite covers in $\ccm$.
That is what we now do, dealing with nonalgebraic simple complex tori first in Proposition~\ref{toruscover} below, and then with (the much easier)  strongly minimal simply connected nonalgebraic compact complex manifolds in Proposition~\ref{simpConnIntern}.
Note that we are now working just in the theory $\ccm$.

\begin{proposition}
\label{toruscover}
Suppose $X$ is a simple nonalgebraic complex torus.
Then $X$ internalises finite covers in $\ccm$.
\end{proposition}

\begin{proof}
See \cite{pillay00} for the basics of the model theory of complex tori.

Suppose $\tp(a/A)$ is stationary and 
almost $X$-internal.
We need to show that it is $X$-internal outright.
There exists $B\supseteq A$ with $a\ind_AB$, and a tuple $c$ from $X(\A')$ such that $a\in\acl(Bc)$.
As $X$ is strongly minimal 
we may replace $c$ by an $\acl$-basis for $c$ over $B$; that is, we may assume $c$ is a $B$-independent $\ell$-tuple of generic points in $X$ over $B$.
We will show that $\tp(ac/B)$ is $X$-internal -- this will suffice as it implies that $\tp(a/B)$ is $X$-internal, and $\tp(a/B)$ is the nonforking extension of $\tp(a/A)$.

Extending $B$ we may assume that $\tp(ac/B)$ is stationary.
Replacing $B$ by the canonical base of $\tp(ac/B)$, we may assume that $B=b$ is a finite tuple.
Let $S:=\loc(b)$ and $Y:=\loc(bac)$.
Note that $\loc(bc)=S\times X^\ell$.
Co-ordinate projections yield the following  commuting diagram of surjective morphisms:
$$\xymatrix{
Y\ar[rd]\ar[rr]^\pi &&S\times X^\ell\ar[dl]\\
&S
}$$
Note that $\pi$ is generally finite-to-one since $a\in\acl(bc)$.
Since $\tp(ac/b)$ is the generic type of the generic fibre of $Y\to S$, to prove that it is $X$-internal it will suffice to prove that, possibly after base change, there is a dominant meromorphic map back from $S\times X^\ell$ to $Y$ over $S$.
This is what we now do.

Consider the Stein factorisation of $\pi$,
$$\xymatrix{
Y\ar[rd]\ar[r]^f &\widetilde Y\ar[d]\ar[r]^{\widetilde{\pi}\ \ \ \ \ \ }&S\times X^\ell\ar[dl]\\
&S
}$$
where $f$ is a bimeromorphism and $\widetilde\pi$ is now finite-to-one everywhere.
Next we take a normalisation of $\widetilde Y$ to get
$$\xymatrix{
&\widehat Y\ar[d]\ar[dr]^{\widehat\pi}\\
Y\ar[rd]\ar[r]^f &\widetilde Y\ar[d]\ar[r]^{\widetilde{\pi}\ \ \ \ \ \ }&S\times X^\ell\ar[dl]\\
&S
}$$
where $\widehat Y$ is normal and $\widehat\pi$ remains a finite morphism.

\begin{claim}
\label{uniform-group}
Possibly after a (finite) base change, the general fibres of $\widehat Y\to S$ are complex tori with a uniformly definable group structure, and $\widehat\pi$ restricted to these fibres is an isogeny.
\end{claim}

\begin{proof}[Proof of Claim~\ref{uniform-group}]
We are grateful to Fr\'ed\'eric Campana for some useful correspondence around these issues.

Fixing a general $s\in S$ we have that $\widehat Y_s$ is normal~\cite[Th\'eor\`eme 2]{banica79} and connected, $X^\ell$ is smooth and $\widehat\pi_s:\widehat Y_s\to X^\ell$ is finite.
It follows by the purity of branch locus theorem~\cite[Theorem~VII.1.6]{scv7} that the ramification locus of $\widehat\pi_s$ is either empty or complex analytic of codimension $1$ in $\widehat Y_s$.
In the latter case the image of this ramification locus would have to be a codimension $1$ complex analytic subset of $X^\ell$, but the fact that $X$ is strongly minimal with $\dim X>1$ makes this impossible (any complex analytic subset of $X^\ell$ has dimension a multiple of $\dim X$).
Hence $\widehat\pi_s$ is an unramified covering.
It follows that $\widehat Y_s$ itself has the structure of a complex torus such that $\widehat\pi_s:\widehat Y_s\to X^\ell$ is an isogeny.

We have not yet proved that this  group structure on $\widehat Y_s$ with respect to which $\widehat\pi_s$ is an isogeny is {\em uniform} in $s$.
Note,
however, that if $\Gamma\leq (X^\ell)^3$ is the graph of the group structure on $X^\ell$, then for each general $s\in S$, $\widehat\pi_s^{-1}(\Gamma)\leq\widehat Y_s^3$ is a possibly disconnected subtorus that is uniformly definable in $s$.
Moreover, since $\widehat\pi_s:\widehat Y_s\to X^\ell$ is finite, $\dim\widehat\pi_s^{-1}(\Gamma)=\dim\Gamma=2\ell \dim X$.
Now, the graph of the (not yet known to be uniformly definable) group operation on $\widehat Y_s$ is of dimension $2\dim\widehat Y_s=2\ell\dim X$ also, and it is a connected subtorus of $\widehat\pi_s^{-1}(\Gamma)$.
It must therefore be the connected component of identity.
This is still not enough because we do not yet know that the identity of $\widehat Y_s$ can be picked out definably in~$s$.
However, it does follow that each connected component of $\widehat\pi_s^{-1}(\Gamma)$ is the graph of {\em some} group structure on $\widehat Y_s$ with respect to which $\widehat\pi_s$ is an isogeny.
Indeed, if $H$ is any such connected component then it is a translate of the graph of the group operation on $Y_s$ by some $(a,b,c)\in\widehat\pi_s^{-1}(\Gamma)$.
So $H$ is itself the graph of a new group operation on $Y_s$, namely $x\oplus y:=x+y-(a+b-c)$.
That $\widehat\pi_s$ remains an isogeny with respect to this new group operation follows from the fact that $\widehat\pi_s(a)+\widehat\pi_s(b)=\widehat\pi_s(c)$.

So it suffices to show that after base change 
the connected components of $\widehat\pi_s^{-1}(\Gamma)$ are uniformly definable.
This we now point out.
Let $Z\subset\widehat Y\times_S\widehat Y\times_S\widehat Y$ be a maximal dimensional irreducible component of $\widehat\pi^{-1}(S\times\Gamma)$ that projects onto $S$.
Let
$$\xymatrix{
Z\ar[r] &T\ar[r]^{\theta} &S
}$$
be the Stein factorisation of $Z\to S$.
So for general $t\in T$, $Z_t$ is a connected component of $\widehat\pi_s^{-1}(\Gamma)$ where $s=\theta(t)$.
Taking the finite base change $\theta:T\to S$, we obtain $\widehat\pi_T:\widehat Y_T\to T\times X^\ell$. Then for general $t \in T$ we have that $\widehat\pi_t^{-1}(\Gamma)$ is connected, and so defines, uniformly in $t$, a group structure with respect to which $\widehat\pi_t$ is an isogeny, as required.
\end{proof}

By the Claim, we may assume that the group structure on the general fibres $\widehat Y_s$ with respect to which $\widehat\pi_s$ is an isogeny is uniformly definable over~$s$.
Let $n$ be the size of the kernel of $\widehat\pi_s$ for $s\in S$ general.
We can define $\rho:S\times X^\ell\to \widehat Y$ over $S$ to be the meromorphic map such that for general $s\in S$, $\rho_s\circ\widehat\pi_s$ is multiplication by $n$ on $\widehat Y_s$, by setting $\rho_s(x) := ny$ for any $y \in \widehat\pi_s^{-1}(x)$.
Finally, looking at the diagram before the statement of Claim~\ref{uniform-group}, we see that $\widehat Y$ admits a dominant meromorphic map onto $Y$ over $S$ since $f$ is bimeromorphic.
We therefore have a dominant meromorphic map from $S\times X^\ell$ to $Y$ over $S$, witnessing that $\tp(a/A)$ is indeed internal to $X$.
\end{proof}

To complete the proof of Theorem~\ref{purity} we need to deal with one more case.

\begin{proposition}
  \label{simpConnIntern}
  Let $X$ be a simply connected nonalgebraic strongly minimal compact complex variety.
  Then $X$ internalises finite covers.
\end{proposition}
\begin{proof}
  The essential point is just that there are no finite covers to internalise.
  This is a uniform version of Lemma~\ref{lem:Bacldcl}.

  As in the first part of the proof of Proposition~\ref{toruscover},
    it suffices to see that for any irreducible $S$,
    if $\widehat\pi : \widehat Y \rightarrow S \times X^\ell$ is a finite cover over $S$ such that for general $s\in S$ the fibre $\widehat Y_s$ is irreducible
    and the map $\widehat\pi_s : \widehat Y_s \rightarrow X^\ell$ is unramified,
    then there is a dominant meromorphic map from $S\times X^\ell$ to $\widehat Y$ over $S$.
    But $X^\ell$ is simply connected since $X$ is,
    so $\widehat\pi_s$ is an isomorphism for general $s\in S$.
  Hence $\widehat\pi^{-1}$ will do.
\end{proof}

\bigskip
\section{Finite-dimensional types}
\label{findim-section}

\noindent
Everything we do in this section could be done in the more general setting of $TA$ for $T$ the theory of an $\omega_1$-compact noetherian topological structure with quantifier elimination and in which  irreducibility is definable.
But we stick to $\ccma$ for the sake of concreteness.

Fix a sufficiently saturated $(\A',\sigma)\models\ccma$.


\begin{definition}
By a {\em $\ccm$-$\sigma$-variety} is meant a pair $(F,G)$ where $F$ is an irreducible closed set in $\A'$ and $G\subseteq F\times F^\sigma$ is an irreducible closed subset whose projections to both $F$ and $F^\sigma$ are surjective and generically finite-to-one.
This data gives rise to the following definable set in $(\A',\sigma)$
$$(F,G)^\sharp:=\{a\in F:(a,\sigma a)\in G\} .$$
\end{definition}

Note that $F^\sigma:=\sigma(F)$ is again a closed set in $\A'$; if $F=X_a$ where $X\subset S\times Z$ is a complex analytic subset of a product of compact complex varieties and $a\in S(\A')$, then $F^\sigma=X_{\sigma(a)}$.

We can associate to a meromorphic dynamical system $(X,f)$ the
$\ccm$-$\sigma$-variety $(F,G)$ where $F=X(\mathcal A')$ and $G$ is the (set
of $\mathcal A'$-points of the) graph of $f$. In fact, a
$\ccm$-$\sigma$-variety should be viewed as the generalisation of a
meromorphic dynamical system where we allow finite-to-finite meromorphic
correspondences in place of dominant meromorphic maps.

$\ccmsigma$-varieties give rise to finite-dimensional types, in the following
sense.

\begin{definition}
\label{def-fd}
Given an inversive (i.e., closed under $\sigma$ and $\sigma^{-1}$) set $A$, and a tuple $a$, we set
$$\dim_{\sigma}(a/A):=\big(\dim\loc((a,\sigma a,\dots,\sigma^na)/A)\big)_{n<\omega}$$
where by $\dim$ is meant the complex-analytic dimension.
We say that $\tpsigma(a/A)$ is {\em finite-dimensional} if $\dimsigma(a/A)$ is eventually constant, in which case we often write $\dimsigma(a/A)=d$ where $d$ is that eventual value.
A definable set is said to be {\em finite-dimensional} if every complete type over an inversive parameter set extending the definable set is finite-dimensional.
\end{definition}

Note that this dimension witnesses independence: for all inversive $B\supseteq A$,
\begin{eqnarray}
\label{indcma}
a\ind^{\ccma}_A B &\iff& \dimsigma(a/B)=\dimsigma(a/A)
\end{eqnarray}
This follows rather easily from~(2) and~(3) of Section~\ref{ta-section}, together with the fact that complex analytic dimension witnesses independence in $\ccm$.

\begin{lemma}
\label{dimind}
Suppose $(F,G)$ is a $\ccmsigma$-variety over inversive $A$.
Then there exists $c\in (F,G)^\sharp$ that is $\ccm$-generic in $F$ over $A$.
Such a point satisfies:
\begin{itemize}
\item[(i)]
$\tpsigma(c/A)$ is finite-dimensional with $\dimsigma(c/A)=\dim F$
\item[(ii)]
for all inversive $B\supseteq A$, $\displaystyle c\ind^{\ccma}_AB$ if and only if $\loc(c/B)=F$.
\end{itemize}
\end{lemma}

\begin{proof}
That we can find $c\in (F,G)^\sharp$ avoiding any particular proper $A$-closed subset of $F$ follows from the axiomatisation of $\ccma$ given in~\cite[Proposition~5.2]{bgh}.
By saturation we get a generic of $F$ over $A$ in $(F,G)^\sharp$.
(Note that this does not use the fact that the projections are finite-to-one, only that they are surjective.)

For~(i), note that $c$ being generic in $F$ and $(c,\sigma(c))\in G$ implies that $\sigma(c)\in\acl(Ac)$.
Applying $\sigma$ repeatedly gives $\sigma^{n}(c)\in\acl(Ac)$ for all $n$.
Hence $\dimsigma(c/A)=(\dim F,\dim F,\dots)$.

Part~(ii) now follows.
Indeed, from~(\ref{indcma}), we have that $\displaystyle c\ind^{\ccma}_AB$ if and only if $\dim\loc\big((c,\sigma c,\dots,\sigma^nc)/B\big)=\dim F$ for all $n$, if and only if $\dim\loc(c/B)=\dim F$.
By irreducibility of $F$, the latter is equivalent to $\loc(c/B)=F$.
\end{proof}

\begin{remark}
The (absolute) irreducibility of $F$ and $G$ in the definition of a $\ccmsigma$-variety is essential.
Suppose $F_1$ and $F_2$ are disjoint irreducible closed sets such that $F_1^\sigma=F_2$ and $F_2^\sigma=F_1$. Let $F=F_1\cup F_2$ and let $G$ be the union of the diagonals in $F_1\times F_1$ and $F_2\times F_2$. Then $G$ projects generically finite-to-one onto both $F$ and $F^\sigma=F$, but $(F,G)^\sharp=\emptyset$.
If in addition $F$ is $A$-definable but $F_1$ is not $A$-definable, then $F$ and $G$ will even be $A$-irreducible.
\end{remark}

We conclude this section by pointing out that $\ccm$-$\sigma$-varieties capture all the finite-dimensional types in $\ccma$.

\begin{lemma}
\label{findim}
Suppose $A$ is $\aclsigma$-closed and $c$ is such that $\dimsigma(c/A)$ is finite.
Then there exists $N\geq 0$ and a $\ccmsigma$-variety $(F,G)$ over $A$, such that $(c,\sigma c,\dots,\sigma^Nc)$ is generic in $F$ over $A$ and contained in $(F,G)^\sharp$.
\end{lemma}

\begin{proof}
Let $N$ and $d$ be such that
$$d=\dim\loc\big((c,\sigma c,\dots,\sigma^Nc)/A\big)=\dim\loc\big((c,\sigma c,\dots,\sigma^{N+1}c)/A\big)$$
Let $\overline c:=(c,\sigma c,\dots,\sigma^Nc)$, $F:=\loc(\overline c/A)$ and $G:=\loc\big((\overline c,\sigma\overline c)/A\big)\subseteq F\times F^\sigma$.
The assumption that $A=\aclsigma(A)$ ensures that $F$ and $G$ are irreducible.
Note that $\dim G=\dim\loc\big((c,\sigma c,\dots,\sigma^{N+1}c)/A\big)=d=\dim F=\dim F^\sigma$.
Moreover, both projections of $G$ are onto as $\overline c$ is generic in $F$ over $A$, and so $\sigma\overline c$ is generic in $F^\sigma$ over $A$ also.
Hence these projections must be generically finite-to-one.
\end{proof}

\begin{remark}
    If furthermore we take $N$ in Lemma~\ref{findim} large enough so that
    the number of realisations of $\tpsigma(\sigma^{N+1}c / A,c,...,\sigma^N c)$ is minimal, then the
    quantifier-free $\ccma$-type of $c$ over $A$ is determined by saying
    that $(c,\sigma c,\dots,\sigma^Nc)\in (F,G)^\sharp$ and $c$ is generic in $F$ over~$A$.
\end{remark}

\bigskip
\section{Difference-analytic jet spaces and the Zilber dichotomy}

\noindent
We now aim to prove the Zilber dichotomy for finite-dimensional minimal types in $\ccma$: either they are one-based or they are nonorthogonal to the fixed field (of the canonical interpretation of $\acfa$ in $\ccma$).
For $\acfa$ itself this was done first by Chatzidakis and Hrushovski in~\cite{acfa1}, but then a much simpler proof was given by Pillay and Ziegler in~\cite{pillayziegler03}.
It is this latter argument, which goes via an appropriate notion of ``jet space" and actually proves something stronger than the dichotomy (namely the Canonical Base Property or the~CBP), that we will follow here, and extend to all of $\ccma$.

In the appendix to this paper we have reviewed the uniformly definable construction of jet spaces in complex analytic geometry.
Given a holomorphic map $\pi:X\to S$ of compact complex varieties, there exists a $\ccm$-definable complex variety $\jet^n (X/S)\to X$ such that for all $x\in X$
$$\jet^n(X/S)_x=\jet^n(X_{\pi(x)})_x=\hom_{\mathbb C}(\mathfrak m_{X_{\pi(x)},x}/\mathfrak m_{X_{\pi(x)},x}^{n+1},\mathbb C)$$
and such that the vector space structure on these fibres is uniformly definable in $\ccm$.
Here $\mathfrak m_{X_{\pi(x)},x}$ denotes the maximal ideal of the local ring of germs of holomorphic functions on the fibre $X_{\pi(x)}$ at the point~$x$.
We point the reader to the appendix for further details.

The uniformity allows us to define jet spaces of nonstandard closed sets.

\begin{definition}
Suppose $F=X_a$ is a closed set in $\ccm$ where $X\subseteq S\times Z$ and $a\in S(\A')$.
Then by the {\em $n$th jet space of $F$}, which we will denote by $\jet^n (F)\to F$, we will mean the restriction of $\jet^n(X/S)(\A')\to X(\A')$ to $X_a$.
\end{definition}

The fibres of $\jet^n(F)\to F$ are therefore uniformly finite dimensional $K$-vector spaces, where recall that $(K,+,\times)$ is the interpretation of the complex field in $\A'$.
Moreover, by the functoriality of the construction of jet spaces, if $F$ is an irreducible closed subset of an irreducible closed set $G$, and $c\in F$, then $\jet^n(F)_c$ is naturally a $K$-linear subspace of $\jet^n(G)_c$.
The main point of this construction is the following proposition which says these jet spaces yield a linearisation of irreducible closed sets.

\begin{proposition}
\label{lineariseccm}
Suppose $S$ and $Z$ are compact complex varieties and $X\subset S\times Z$ is a subvariety.
There exists $n>0$ such that for all $a,b\in S(\mathcal A')$, if $X_a$ and $X_b$ are irreducible, pass through some $c\in Z(\mathcal A')$, and have $\jet^n(X_a)_c=\jet^n(X_b)_c$ as $K$-subspaces of $\jet^n(Z)_c$, then $X_a=X_b$
.
\end{proposition}

\begin{proof}
The statement that is claimed by the proposition (for fixed $S,Z,X$) is first-order in the language of $\ccm$ (using the fact that irreducibility is definable), and hence it suffices to prove that it is true in the standard model.
Now, working in the standard model, by Lemma~\ref{jetsdetermine}, it is the case that for any $s,s'\in S(\mathcal A)$ if $X_s$ and $X_{s'}$ are irreducible and pass through $z\in Z(\mathcal A)$ with $\jet^n(X_s)_z=\jet^n(X_{s'})_z$ for {\em all} $n\geq 0$, then $X_s=X_{s'}$.
To see that we can uniformly bound $n$ is a straightforward compactness argument:
For each $n>0$, let $\phi_n(s,s',z)$ be the formula expressing that $X_s$ and $X_{s'}$ are irreducible, pass through $z$, satisfy $\jet^n(X_s)_z=\jet^n(X_{s'})_z$ but $X_s\neq X_{s'}$.
We are using here the definability of irreducibility and the fact that the jet spaces are uniformly definable in order to write down $\phi_n$.
Note that $\phi_n(s,s',z)\rightarrow\phi_m(s,s',z)$ for all $m\leq n$.
What we need to prove is that for some $n$, $\phi_n$ is not satisfiable in $\A$.
But by $\omega_1$-compactness of $\A$, if each $\phi_n$ were satisfiable then we would find $s,s',z$ such that $X_s\neq X_{s'}$ but $\jet^n(X_s)_z=\jet^n(X_{s'})_z$ for all $n\geq 0$, contradicting what was previously established.
\end{proof}

Now we pass to $\ccma$ and define jet spaces there using the nonstandard complex analytic jet spaces described above, in very much the same way that Pillay and Ziegler use algebraic jet spaces to define jet spaces for finite-dimensional difference-algebraic varieties in~\cite[$\S$4]{pillayziegler03}.

Work in a sufficiently saturated model $(\A',\sigma)\models\ccma$, and 
fix a $\ccmsigma$-variety $(F,G)$ defined over $B=\acl(B)$.
We can arrange things so that $F=X_a$ and $G=W_{(a,\sigma(a))}\subset X_a\times X_{\sigma(a)}$ where
\begin{itemize}
\item
$a$ is a tuple from $B$ that is a generic point of a compact complex variety $S$,
\item
$X\to S$ is a {\em fibre space}; that is, the general fibres are irreducible,
\item
$(a,\sigma(a))$ is a generic point of a subvariety $T\subseteq S^2$, and
\item
$W\subseteq X^2$ is a subvariety such that the induced map $W\to T$ is a fibre space.
\end{itemize}
The fact that $G$ projects generically finite-to-one onto $F$ and $F^\sigma$ implies, by definability of dimension, that in the standard model, for general $t=(s_1,s_2)\in T(\mathcal A)$, $W_t\subseteq X_{s_1}\times X_{s_2}$ projects generically finite-to-one onto each co-ordinate.
It follows that for any $n>0$ and  for general $x_1\in X_{s_1}(\mathcal A)$ and $x_2\in X_{s_2}(\mathcal A)$, $\jet^n(W_t)_{(x_1,x_2)}$ is the graph of a $\mathbb C$-linear isomorphism from $\jet^n(X_{s_1})_{x_1}$ to $\jet^n(X_{s_2})_{x_2}$.
(For this latter property of jet spaces, see for example Lemma~5.10 of~\cite{paperA}, which works in the algebraic setting but goes through in our analytic setting.)
As this is a definable property we get for $c\in (F,G)^\sharp$, with $c$ generic in $F$ over $B$, that $\jet^n(G)_{(c,\sigma(c))}$ is the graph of a ($\ccm$-definable) $K$-linear isomorphism $g:\jet^n(F)_c\to\jet^n(F^\sigma)_{\sigma(c)}$.

\begin{definition}
Suppose $(F,G)$ is a $\ccmsigma$-variety over $B=\acl(B)$, and $c\in (F,G)^\sharp$ is generic in $F$ over $B$.
By the {\em $n$th jet space of $(F,G)^\sharp$ at $c$} we mean the $Bc$-definable $\fix(K,\sigma)$-vector subspace
$$\jet^n(F,G)^\sharp_c:=\big(\jet^n(F)_c,\jet^n(G)_{(c,\sigma(c))}\big)^\sharp=\{v\in \jet^n(F)_c:\sigma(v)=g(v)\}$$
\end{definition}

Note that $g^{-1}\sigma:\jet^n(F)_c\to \jet^n(F)_c$ endows $\jet^n(F)_c$ with the structure of a {\em $\sigma$-module} over the difference field $(K,\sigma)$;
namely $g^{-1}\sigma(rv)=g^{-1}\big(\sigma(r)\sigma(v)\big)=\sigma(r)g^{-1}\sigma(v)$ for all $r\in K$ and $v\in \jet^n(F)_c$.
This uses the fact that scalar multiplication on $\jet^n(F)_c$ is obtained by the restriction of the $0$-definable scalar multiplication on $\jet^nX$, and hence commutes with $\sigma$.
Now $\jet^n(F,G)^\sharp_c$ is the fixed set of this $\sigma$-module, which therefore is naturally a $\fix(K,\sigma)$-vector subspace.
Moreover, by~\cite[Lemma~4.2(ii)]{pillayziegler03}, the fact that $(K,\sigma)$ is existentially closed implies that
$$\jet^n(F)_c=\jet^n(F,G)^\sharp_c\otimes_{\fix(K,\sigma)}K$$
In particular, the dimension of $\jet^n(F,G)^\sharp_c$ as a $\fix(K,\sigma)$-vector space is equal to $\dim_K\jet^n(F)_c$.

\medskip

We will use the difference-analytic jet spaces defined above to prove a canonical base property for $\ccma$.
Let us first recall briefly what canonical bases in simple theories are.
In a simple theory an {\em amalgamation base} is a type $p\in S(A)$ for which type-amalgamation holds; that is if $B_1$ and $B_2$ are extensions of $A$ that are independent over $A$, and $q_i\in S(B_i)$ is a nonforking extension of $p$ for $i=1,2$, then $q_1\cup q_2$ is consistent and does not fork over $A$.
{\em Parallelism} among amalgamation bases is the transitive closure of the relation of having a common nonforking extension.
Given an amalgamation base $p$, the {\em canonical base of $p$}, denoted by $\cb(p)$, can be characterised by the following property: an automorphism of the universe fixes $\cb(p)$ if and only if it fixes the parallelism class of $p$ setwise.
For more details, and further properties of canonical bases that we will use, we suggest~\cite[Section~4.3]{kim}.

The {\em canonical base property} (CBP) refers to a condition that a theory may or may not satisfy, which tells us something about the type of $\cb(p)$ over a realisation of~$p$.
We will not make the CBP precise in general here, but we do articulate it in our context as Theorem~\ref{cbp} below.

Because $\ccma$ is supersimple, it follows by~\cite{bpw} that types over $\aclsigma^{\eq}$-closed sets are amalgamation bases, and that the canonical base of an amalgamation base exists as a (possibly infinite) tuple of imaginary elements.
(In general simple theories canonical bases exist as ``hyperimaginary" elements.)
We will therefore have to pass sometimes to $(\mathcal A',\sigma)^{\eq}$.

The following theorem should be viewed as a generalisation of~\cite[Theorem~1.2]{pillayziegler03} from $\acfa$ to $\ccma$.
Indeed, our proof is modelled on the $\acfa$ case.

\begin{theorem}[CBP for finite-dimensional types in $\ccma$]
\label{cbp}
Suppose $(F,G)$ is a $\ccm$-$\sigma$-variety defined over $B=\aclsigma(B)$, and $c\in(F,G)^\sharp$ is generic in $F$ over $B$.
Let $B_1\supseteq B$ be an $\aclsigma^{\eq}$-closed set of parameters.
If $e=\cbsigma(c/B_1)$ then $\tpsigma(e/Bc)$ is almost $\fix(K,\sigma)$-internal.
\end{theorem}

\begin{proof}
Let $B_1'\supseteq B$ be the intersection of $B_1$ with the real sorts of $(\A',\sigma)$.
By geometric elimination of imaginaries (see~(4) on page~\pageref{gei}), $B_1\subseteq\aclsigma^{\eq}(B_1')$.

Let $F_1=\loc(c/B_1')$ and let $d$ be a canonical parameter for $F_1$ in $\ccm$.
We first claim that $d$ and $e$ are interalgebraic over $B$ in $\ccma$.
Let $\alpha$ be an automorphism of $(\A',\sigma)$ that is the identity on $Be$.
Then $\alpha(c)\in (F,G)^\sharp$ is generic in $F$ over $B$ also, and by properties of canonical bases $\tpsigma\big(\alpha(c)/\alpha(B_1)\big)$ is parallel to $\tpsigma(c/B_1)$.
So $\tpsigma\big(\alpha(c)/\alpha(B_1')\big)$ is parallel to $\tpsigma(c/B_1')$ also.
It follows by Lemma~\ref{dimind} that $\loc\big(\alpha(c)/\alpha(B_1')\big)=\loc(c/B_1')$.
That is, $\alpha(F_1)=F_1$, and so $\alpha(d)=d$.
This shows that $d\in\dclsigma(Be)$.
On the other hand, Lemma~\ref{dimind} also implies that $\displaystyle c\ind^{\ccma}_dB_1'$, and so $e=\cbsigma(c/B_1)\in\aclsigma^{\eq}(d)$.

So it suffices to show that $\tpsigma(d/Bc)$ is almost $\fix(K,\sigma)$-internal.
In fact, we will show it is $\fix(K,\sigma)$-internal.

Let $n>0$ be as given by Proposition~\ref{lineariseccm} applied to compact complex varieties $X\subset S\times Z$ where $F_1=X_a$ for $a\in S(\A')$.
Let $A$ be a $\fix(K,\sigma)$-basis for $\jet^n(F,G)^\sharp_c$, chosen so that $\displaystyle A\ind^{\ccma}_{Bc}d$.
We show that $d\in\dclsigma\big(B,c,A,\fix(K,\sigma)\big)$.

Setting $G_1=\loc\big((c,\sigma(c))/B_1\big)$, we have that $(F_1,G_1)$ is a $\ccmsigma$-variety, and so $\jet^n(G_1)_{(c,\sigma(c))}$ induces on $\jet^n(F_1)_c$ the structure of a $\sigma$-module over $(K,\sigma)$.
Now $\jet^n(G_1)_{(c,\sigma(c))}$ is a $K$-subspace of $\jet^n(G)_{(c,\sigma(c))}$, so that $\jet^n(F_1)_c$ is a $\sigma$-submodule of $\jet^n(F)_c$.
Hence $\jet^n(F_1,G_1)^\sharp_c$ is a $\fix(K,\sigma)$-subspace of $\jet^n(F,G)^\sharp_c$.

Let $\alpha$ be an automorphism of $(\mathcal A',\sigma)$ that fixes the sets $B,c,A,\fix(K,\sigma)$ pointwise.
Then $\alpha$ fixes all of $\jet^n(F,G)^\sharp_c$ pointwise, and hence all of $\jet^n(F_1,G_1)^\sharp_c$.
But, $\jet^n(F_1)_c=\jet^n(F_1,G_1)^\sharp_c\otimes_{\fix(K,\sigma)}K$, so that  $\alpha$ preserves $\jet^n(F_1)_c$ setwise.
That is, $\jet^n(F_1)_c=\jet^n(F_1^\alpha)_c$.
So by Proposition~\ref{lineariseccm}, $F_1=F_1^\alpha$.
As $d$ is a canonical parameter for $F_1$ in $\ccm$, we have that $\alpha(d)=d$, as desired.
\end{proof}

As was first observed by Pillay~\cite{pillay02}, from such a canonical base property one can deduce a Zilber dichotomy statement.
Ours will apply to finite-dimensional minimal types.
Recall that a type in a simple theory is {\em minimal} if it is of $SU$-rank one, that is, every forking extension is algebraic.

\begin{corollary}[Zilber dichotomy for finite-dimensional minimal types in $\ccma$]
\label{zd}
Suppose $B$ is an $\aclsigma$-closed set and $p(x)\in S(B)$ is a finite-dimensional minimal type.
Then either $p$ is one-based or it is almost internal to $\fix(K,\sigma)$.
\end{corollary}

\begin{proof}
We follow the argumentation of~\cite[Corollary~6.19]{paperC}.

Suppose $p(x)$ is not one-based.
Then there exists a finite tuple $c$ of realisations of $p(x)$ and a model $M\supseteq B$ such that $e:=\cbsigma(c/M)\notin\aclsigma(Bc)$.
Note that $\tpsigma(c/B)$ is also finite-dimensional so that, possibly after replacing $c$ by $(c,\sigma c,\dots,\sigma^nc)$ for some $n$, we may assume by Lemma~\ref{findim} that there is a $\ccmsigma$-variety $(F,G)$ over $B$ and that $c$ is generic in $F$ over $B$ and contained in $(F,G)^\sharp$.
Hence, by Theorem~\ref{cbp} (with $B_1=M$), we have that $\tpsigma(e/Bc)$ is almost $\fix(K,\sigma)$-internal.
So we have $e\in\aclsigma(A,B,c,d)$ for some tuple $d$ from $\fix(K,\sigma)$ and some $\displaystyle A\ind^{\ccma}_{Bc}e$.
It follows that $\displaystyle e\mathbb \nind^{\ccma}_{ABc}d$.
Now $e$ is in the definable closure of a finite set of realisations of $p(x)$; this is because $e$ is in the definable closure of a finite part of a Morley sequence in $\tpsigma(c/M)$
, and $c$ is a tuple of realisations of $p(x)$.
Since $SU(p)=1$, we can find $ABc$-independent realisations $a_1,\dots,a_m$ of $p(x)$ such that $e\in\aclsigma(ABc,a_1,\dots,a_m)$.
Hence $\displaystyle (a_1,\dots,a_m)\mathbb \nind^{\ccma}_{ABc}d$.
But then for some $i<m$, $a_{i+1}\in\aclsigma(a_1\dots a_iABcd)$, witnessing that $p(x)$ is almost internal to $\fix(K,\sigma)$.
\end{proof}

\begin{remark}
On the face of it the above corollary is only about real types, as per our conventions we do not work in $(\A',\sigma)^{\eq}$.
But because $\ccma$ admits geometric elimination of imaginaries, and because both one-basedness and almost internality to $\fix(K,\sigma)$ are preserved by interalgebraicity, we get the Zilber dichotomy for all finite-dimensional minimal types in $(\A',\sigma)^{\eq}$ also.
\end{remark}

The Zilber dichotomy result of Corollary~\ref{zd} applies only to finite-dimensional minimal types.
We give an example to show that not all minimal types are finite-dimensional.
Whether a Zilber dichotomy holds for these types we leave as an open question;
we only give an example of a trivial infinite-dimensional minimal type.

\begin{example}
  \label{infDimTrivMin}
  Let $X$ be a simply connected strongly minimal compact complex variety
    of dimension greater than 1
    and with trivial automorphism group;
    for example, almost all generic K3 surfaces have these properties
    \cite[Theorem~3.6]{MacriStellariAMsK3}.

  By Theorem~\ref{purity}, $X(\mathcal A')$ is stably embedded in $(\A',\sigma)$ with induced structure $L^X_{\sigma}$.
  By \cite[Proposition~2.3]{MoosaPillayAleph0CatCCM}, the complex analytic structure on $X$, namely $\big(X(\mathcal A'),L^X\big)$, is just the structure of equality.
  So the full induced structure on $X(\mathcal A')$
    is that of an infinite set with a generic permutation $\sigma$,
    and with distinguished $\sigma$-fixed points for the complex points $X(\A)$.
    It follows that the type in $\ccma$ of an aperiodic point of $X$,
    \[ p := \{x \in X\}\cup\{ x \neq \sigma^n x \;|\; n\in\mathbb N \} ,\]
    is complete, minimal, and has trivial geometry.
  If $b \vDash p$, then $\loc(b,\ldots,\sigma^{n-1}b) = X^n$,
    so $p$ has infinite $\dimsigma$.

  We may also note that the types of periodic points of $X$ provide examples of trivial minimal types of finite dimension which are orthogonal to $\acfa$. Indeed, the discussion above shows that $X$ is 
  {\em fully orthogonal} to $\acfa$ in the sense that every tuple from $X$ is independent of every tuple from $K$ over any parameters.
\end{example}

\bigskip
\section{Minimal one-based types}

\noindent
The Zilber dichotomy theorem, Corollary~\ref{zd}, tells us that there are no new non-one-based finite-dimensional minimal types in $\ccma$; they all come from $\acfa$ and in fact from the fixed field in $\acfa$.
In Example~\ref{infDimTrivMin}, we saw that there are new trivial minimal types,
both of finite and of infinite dimension.
It remains to consider the minimal one-based nontrivial types.
The following example shows that new ones do occur in $\ccma$, of finite dimension.

Fix $(\A',\sigma)\models\ccma$ sufficiently saturated.

\begin{example}
\label{mnt1bexample}
There are definable groups in $(\A',\sigma)$ that are finite-dimensional, minimal, one-based, and fully orthogonal to $\acfa$.
\end{example}

\begin{proof}
If $X$ is any nonalgebraic simple complex torus, then $\fix\big(X(\A'),\sigma\big)$ will be such an example.
In fact, we will describe precisely which definable subgroups of $X(\A')$ will have the desired properties.

First of all, $X$ is a one-based stable group in $\ccma$.
Indeed, it is a one-based stable group in $\ccm$, and so $\big(X(\mathcal A'),L^X_{\sigma}\big)$ is a one-based stable group by~\cite{gr}, and by Theorem~\ref{purity} the latter is the full structure induced on $X(\A')$ by $(\mathcal A',\sigma)$.
Secondly, as $X$ is fully orthogonal to $\mathbb P$ in $\ccm$ it follows from the characterisation of independence in $\ccma$ -- see~(3) on page~\pageref{ss} -- that $X$ and $\mathbb P$ remain fully orthogonal in $\ccma$.
Hence, one-basedness and full orthogonality to $\acfa$ will come for free; what we require is simply a description of all the finite-dimensional definable subgroups of $X(\A')$ of $SU$-rank one.
Hrushovski's arguments in~\cite[$\S$4.1]{maninmumford}, which are written for simple abelian varieties in $\acfa$ but which work equally well for simple complex tori in $\ccma$, show that the definable subgroups of $X(\A')$ of $SU$-rank one are precisely those of the form $\ker(g)$ where $g$ is an element of $\endo(X)[\sigma]$ and is left-irreducible in $\endo_{\mathbb Q}(X)[\sigma,\sigma^{-1}]$.
Here, $\endo(X)$ denotes the (noncommutative unitary) ring of holomorphic homomorphisms from $X$ to itself, and $\endo_{\QQ}(X):=\endo(X)\otimes_{\ZZ}\QQ$.
Note that $\ker(g)$ will be finite-dimensional.
Taking $g=\sigma-1$ yields the example of $\fix\big(X(\A'),\sigma\big)$.
\end{proof}

The following proposition will essentially say that these are the only nontrivial one-based minimal types of finite dimension that do not come from $\acfa$.
However, in order to express this, we need to talk about groups definable over parameters in $(\A',\sigma)$, and in particular those that correspond to complex tori.
The usual proof of the Weil - van den Dries - Hrushovski theorem goes through for $\ccm$, so that every group interpretable in $\A'$ can be endowed with the structure of an {\em $\A'$-meromorphic group} in the sense of~\cite[Definition~4.3]{ams}; it is the natural nonstandard analogue of meromorphic group.
We will say (somewhat oddly) that an $\A'$-meromorphic group is {\em $\A'$-compact} if as an $\A'$-manifold it is the image of a closed set in $\A'$ under an $\A'$-holomorphic map (again the notions of $\A'$-manifold and $\A'$-holomorphic are the natural nonstandard ones, see~\cite[Definition~4.1]{ams}).
An $\A'$-compact group appears as the generic fibre of a definable family of meromorphic groups that are individually, though not necessarily uniformly, definably isomorphic to complex tori.
In particular,  $\A'$-compact groups are commutative.
One would like to think of them as ``nonstandard complex tori", except that  outside of the essentially saturated context there are some subtleties (and open questions) about doing so, see the discussion in the Introduction to~\cite{ams}.

One other point of terminology before we state the proposition.
Until now in this paper we have used the term ``generic" (for points or types) only in the Zariski topology sense of the theory $\ccm$.
But now we want to talk about {\em generic types} of groups definable in $\ccma$, and we will mean this in the sense of groups definable in simple theories; see for example~\cite[Chapter~7]{kim}.

\begin{remark}
Let $G$ be a definable group in $(\A',\sigma)$ that is finite-dimensional (see Definition~\ref{def-fd}). Then a type is generic if and only if it is of maximal $\dimsigma$.
\end{remark}

\begin{proof}
In this proof, for a finite tuple $a$ from $G$ and a parameter set $C$ we set $d(a/C)=\dimsigma(a/\acl_\sigma(C))$ (which we think of as an element of $\NN$).
The function $d$ satisfies the following properties:

\begin{itemize}
\item[(1)] $d$ witnesses independence: if $C\subseteq B$ and $a$ are given, then $d(a/B)\leq d(a/C)$, where equality holds if and only if $\displaystyle a\ind^{\ccma}_C B$.
\item[(2)] $d(h/C,g)=d(g\cdot h/C,g)=d(h\cdot g/C,g)$ for all $g,h\in G$ (translation invariance).
\end{itemize}

Note that 
(1) is \eqref{indcma} on page \pageref{indcma}, and (2) follows from the fact that $\dimsigma$ is invariant 
under definable bijections.

The statement now follows formally from properties (1)-(2). Suppose for convenience that $G$ is 0-definable. Let $a\in G$ be such that $d(a)$ is maximal, and let $g\in G$ be generic 
in $G$ over $a$. As $\displaystyle a\ind^{\ccma} g$, we get $d(a)=d(a/g)=d(g\cdot a/g)$, by (1) and (2). Thus $\displaystyle g\cdot a \ind^{\ccma} g$ by maximality of $d(a)$ and (2). 
But $g\cdot a$ is generic (by genericity of $g$ over $a$), and so $a=g^{-1}\cdot g\cdot a$ is generic as well. Moreover, as $g\cdot a$ is independent
from $a$, we conclude that $d(g)=d(g/a)=d(g\cdot a/a)=d(g\cdot a)$, so $g$ is of maximal dimension.
\end{proof}

\begin{proposition}
\label{mnt1b}
Suppose $p(x)\in S(B)$ is a finite-dimensional minimal nontrivial one-based type in $(\A',\sigma)$ that is orthogonal to the projective line.
There exist a commutative simple nonalgebraic $\A'$-compact $\A'$-meromorphic group $G$ and a finite dimensional quantifier-free definable subgroup~$H$ of $SU$-rank one such that $p(x)$ is nonorthogonal to (all) the generic types of $H$.
\end{proposition}

\begin{remark}
In the case when $B=\emptyset$ and the sort of $p(x)$ is of
K\"ahler-type, the group $G$ in the conclusion can be taken to be (the
interpretation in $\mathcal A'$ of a) simple nonalgebraic complex torus.
\end{remark}

\begin{proof}[Proof of~\ref{mnt1b}]
This is just the analogue for $\ccma$ of the corresponding fact about minimal nontrivial one-based types in $\acfa$ which is due to Chatzidakis and Hrushovski (see \cite[Theorem~5.12]{acfa1}).
Indeed, their proof goes through in this setting. We will give the proof in detail.

By finite dimensionality, replacing $p(x)$ with something interdefinable with it, we can assume that if $a\models p(x)$ then $\sigma(a)\in\acl(Ba)$. By non-triviality, we find realisations $a,b,x,z$ of $p$ which are 
dependent and such that any triple is independent. By modularity and geometric elimination of imaginaries for $\ccma$, we find tuples $c,w$ and $y$ of $\SU$-rank 1 such that 
\begin{itemize}
\item $\aclsigma(B,a,b)\cap\aclsigma(B,x,z)=\aclsigma(B,c)$;
\item $\aclsigma(B,a,z)\cap\aclsigma(B,b,x)=\aclsigma(B,w)$, and 
\item $\aclsigma(B,a,x)\cap\aclsigma(B,b,z)=\aclsigma(B,y)$.
\end{itemize}

By finite dimensionality, enlarging $c$ if necessary, we may assume that $\sigma(c)\in\acl(B,c)$; similarly for $w$ and $y$. The tuple 
$(a,b,c,x,y,z,w)$ forms an abelian group configuration (see Theorem~\ref{T:AbGpConf}) 
with respect to the theory $\ccm$. 
Let $(M,\sigma)\elex\A'$ be sufficiently saturated. We may assume that $a,b,c,x,y,z,w\ind_B M$. The abelian group configuration theorem (Theorem \ref{T:AbGpConf}) yields a connected abelian $\A'$-meromorphic group $G$ defined over $M$ such that there are independent generics $a',b'$ in $G$ over $M$ (in the sense of $\ccm$) which satisfy $\acl(M,a)=\acl(M,a')$ and $\acl(M,b)=\acl(M,b')$ and $\acl(M,c)=\acl(M,c')$, where $c'=a'+b'$.

Let $q_0:=\qftp_\sigma(a'/M)$. This is an $M$-definable (partial) type, as quantifier free formulas are stable in $\ccma$. Let $H:=\Stab({\mathbf q_0})$, where ${\mathbf q_0}$ is the unique non-forking 
extension of $q_0$ to a quantifier-free complete type over $\A'$.

\begin{claim}\label{C:qfstab}
\begin{enumerate}
\item The subgroup $H$ is definable by a quantifier-free formula.
\item $\SU(H)=1=\SU(q_0)$, $H$ is quantifier-free connected, and any non-algebraic type in $H$ is non-orthogonal to any completion of $q_0$. 
\end{enumerate}
\end{claim}

\begin{proof}[Proof of Claim \ref{C:qfstab}]
First note that $H$ is finite-dimensional; indeed, any element of $H$ is of the form $a-b$ for $a$ and $b$ realising $q_0$, and $q_0$ is finite-dimensional.
Now observe that $H$ is the intersection of all $\phi$-stabilisers, where $\phi=\phi(x;y,z)=\psi(y+x,z)$ with $\psi(w,z)=\widetilde{\psi}(w,\sigma(w),\ldots,\sigma^m(w),z)$ for some $\L$-formula $\widetilde{\psi}$. 
By the \emph{$\phi$-stabiliser} we mean $\Stab_{\phi}({\mathbf q_0}):=\{g\in G\mid \forall y,z (\delta(y+g,z) \leftrightarrow\delta(y,z))\}$, where $\delta(y,z)$ is the $\phi$-definition of $q_0$ (cf. \cite[proof of Lemma 1.6.16(i)]{Pil96}).

Every such $\phi$-stabiliser is a quantifier-free definable subgroup of $G$. Indeed, an element $h\in G$ stabilises the $\phi$-type of $a'$ over $M$ if and only if $\tilde{h}=(h,\sigma(h),\ldots,\sigma^m(h))$ stabilises the $\widetilde{\phi}$-type of 
$\tilde{a'}=(a',\sigma(a'),\ldots,\sigma^m(a'))$ over $M$ in the 
$\L$-definable group $\widetilde{G}=G\times\sigma(G)\times\cdots\times\sigma^m(G)$. Here, $\widetilde{\phi}(\tilde{x};\tilde{y},z):=\widetilde{\psi}(\tilde{y}+\tilde{x},z)$, 
where the 
group operation  is the one in $\widetilde{G}$. 

As $\ccma$ is supersimple, there is no infinite descending chain of 
definable groups $(H_i)_{i\in\omega}$ such that $(H_i:H_{i+1})$ is infinite for all $i$. 
To prove (1), it is thus enough to show that there is no infinite strictly descending 
chain of quantifier-free definable subgroups $(H_i)_{i\in\omega}$ of $G$ with $\dim_\sigma(H_0)<\infty$ such that $(H_i:H_{i+1})$ is finite for all $i$. This is what we will 
prove now.

First observe that there is $n$ such that $\sigma^n(h)\in \acl(M,h,\ldots,\sigma^{n-1}(h))$ for all $h\in H_0$, so we may assume that $\sigma(h)\in \acl(M,h)$ for 
all $h\in H_0$, by replacing $H_0$ with a set which is in definable bijection with it.
In particular,  $h\in H_0$ is generic in $H_0$ over $M$ if and only if $\dim(h/M)$ (which by assumption is equal to $\dim_\sigma(h/M)$) is maximal.
We claim there are at most countably many complete quantifier-free types 
extending $x\in H_0$ which are of maximal dimension~$d$.
Indeed,
for any $h\in H_0$ there are natural numbers $N$ and $m$ such that for all $k\geq N$, $\tp(\sigma^{k+1}(h)/M,h,\ldots,\sigma^{k}(h))$ is of multiplicity $m$, and so 
$\tp(\sigma^{k+1}(h)/M\sigma(h),\ldots,\sigma^{k}(h))\vdash\tp(\sigma^{k+1}(h)/M,h,\ldots,\sigma^{k}(h))$ for all $k\geq N$.
Thus $\tp(h,\ldots,\sigma^N(h)/M)\vdash\qftp_\sigma(h/M)$ in this case.
But for every $N$ there are only finitely many possibilities for $\tp(h,\ldots,\sigma^N(h)/M)$ of dimension~$d$ as they will all be $\ccm$-generic types of the Zariski closure of $\{(h,\sigma h,\dots,\sigma^Nh):h\in H_0\}$.

Now, a strictly decreasing sequence $(H_i)_{i\in\omega}$ of quantifier-free definable groups with $(H_i:H_{i+1})$ finite for all $i$ would give rise to at least continuum many generic quantifier-free types in $H_0$ over $M$, 
and hence no such sequence exists.
This completes the proof of~(1).

\smallskip

To prove (2), we argue as in \cite{Zie06}. Let $b''\models \tp_{\sigma}(b'/M,c')$ such that $\displaystyle b''\ind^{\ccma}_{M,c'}a',b'$. As $\displaystyle b'\ind^{\ccma}_M c'$, we get $\displaystyle b''\ind^{\ccma}_M c'$, and thus $\displaystyle b''\ind^{\ccma}_M a',b'$ by transitivity.
Since $\displaystyle a'\ind^{\ccma}_M b'$, we get $\displaystyle a'\ind^{\ccma}_M b',b''$ and thus $\displaystyle a'\ind^{\ccma}_M b'-b''$.  As $\tp_\sigma(b',c'/M)=\tp_\sigma(b'',c'/M)$, we get $\tp_\sigma(b',a'/M)=\tp_\sigma(b'',c'-b''/M)$, therefore in particular 
$\tp_\sigma(a'/M)=\tp_\sigma(a'+b'-b''/M)$, showing that $b'-b''\in H$.

Now $$\SU(b'-b''/M)\geq\SU(b'-b''/M,b'')=\SU(b'/M,b'')=\SU(b'/M)=1,$$
and so $b'-b''$ is non-algebraic. On the other hand, 
\begin{eqnarray*}
\SU(b'-b''/M) &=& \SU(b'-b''/M,a') = \SU(a'+b'-b''/M,a')\\
& \leq &\SU(a'+b'-b''/M) = \SU(a'/M) = 1.
\end{eqnarray*}


 It follows that 
$(a',b'-b'',a'+b'-b'')$ is a pairwise independent triple and $\tp_{\sigma}(b'-b''/M)$ is minimal. 

Let $r_0:=\qftp_\sigma(b'-b''/M)$. Note that for any $a_0\models q_0$ and any $d_0\models r_0$ with $\displaystyle d_0\ind^{\ccma}_M a_0$ we get $\qftp_\sigma(a_0,d_0/M)=\qftp_\sigma(a',b'-b''/M)$, 
and so
$$\qftp_\sigma(a_0,d_0,a_0+d_0/M)=\qftp_\sigma(a',b'-b'',a'+b'-b''/M),$$ so in particular $a_0,d_0,a_0+d_0$ is pairwise independent over $M$, and each singleton is 
algebraic over $M$ and the other two (in $\ccm$). Statement (2) now follows. Note that $r_0$ is the only generic quantifier-free type of $H$, as any $h\in H$ which is generic over $M$ 
is of the form $a_1-a_2$ for some $(a_1,a_2)\models q_0^{(2)}$.
\end{proof}

We also infer from (2) that $\dim(h/M)=\dim(a'/M)=\dim(G)$ for $h$ generic in $H$ over $M$. 
We claim that $G$ is simple (as an $\A'$-meromorphic group). To prove this, suppose $N$ is a proper infinite $\A'$-meromorphic subgroup of $G$ which is defined over $M$. Then $\dim(N)<\dim(G)$ as well 
as $\dim(G/N)<\dim(G)$, where $G/N$ is the quotient group. As $H$ is of $\SU$-rank 1 and quantifier-free connected, any proper quantifier-free definable subgroup of $H$ is finite. Thus, either $N\supseteq H$, or $N\cap H$ is finite. In the 
first case, we get $\dim(N)\geq\dim(G)$, and in the second case, we get $\dim(G/N)\geq\dim(H/N\cap H)=\dim(G)$. In both cases, this is a contradiction.

Note that as $p$ is assumed to be orthogonal to the projective line, $G$ must be nonalgebraic.
It remains to see that $G$ is $\A'$-compact.
Meromorphic groups in $\A$ were given a Chevalley-type characterisation in~\cite{mergroups}, this was extended to strongly minimal $\A'$-meromorphic groups in~\cite{ams}, and then to all $\A'$-meromorphic groups in~\cite{scanlonmergroups}.
That characterisation says that $G$ is the extension of an $\A'$-compact group by a linear algebraic group. Nonalgebraicity and simplicity thus force $G$ to be $\A'$-compact.
\end{proof}

\appendix

\bigskip
\section{If $TA$ exists then $T^{\PP}A$ exists}

\begin{fact}
\label{Fact:TPA}
Suppose $T=T^{\eq}$ is a complete (multi-sorted) stable theory admitting quantifier elimination in a language~$L$, $\PP$ is a set of sorts in $T$, $L^{\PP}$ is the language consisting of a predicate symbol for each $L$-formula whose variables belong to sorts in $\PP$, and $T^{\PP}=\Th_{L^{\PP}}(\PP)$.
Assume $TA$ exists.
Then the following hold:
\begin{itemize}
\item[(i)] $T^{\PP}A$ exists.
\item[(ii)]If $(M,\sigma)\models TA$, then $(\PP(M),\sigma)\models T^{\PP}A$.
\item[(iii)] Conversely, for every $(P,\sigma)\models T^{\PP}A$ there is a model $(M,\sigma)\models TA$ such that $(\PP(M),\sigma)\elres(P,\sigma)$.
\end{itemize}
\end{fact}

\begin{proof}
This follows from the work of Chatzidakis and Pillay \cite[\S3]{ChPi98}. As we could not find it explicitly stated, we give a proof. 
As replacing $\PP$ by $\dcl(\PP)$ does not change the truth value of the statements (i), (ii) or (iii), we may assume that $\PP=\PP^{\eq}$, i.e. that $T^{\PP}$ eliminates imaginaries as well.

Let $\kappa={\mid \!T\!\mid ^+}$, and let $\mathcal{C}_{T,\sigma}$ 
be the class of $L_\s$-structures $(A,\sigma)$, where $A$ is an algebraically closed subset of a model of $T$ of cardinality $<\kappa$ and $\sigma$ is an elementary permutation of $A$. A model $(M,\sigma)$ of $T_\sigma$ is said to be \emph{$\kappa$-generic} if whenever $(A,\sigma)\subseteq(B,\sigma)$ are elements of $\mathcal{C}_{T,\sigma}$, any embedding of $(A,\sigma)$ into $(M,\sigma)$ extends to an 
embedding of $(B,\sigma)$ into $(M,\sigma)$. 

The following results are shown in \cite[3.4 and 3.5]{ChPi98}. (The proof in loc.\ cit.\ is given for countable $T$, but it works identically in arbitrary cardinality.)
\begin{itemize}
\item[(1)] $TA$ exists if and only if every $(M_0,\sigma_0)\models T_\sigma$ embeds into some $\kappa$-generic $(M,\sigma)\models T_\sigma$ which is $\kappa$-saturated.
\item[(2)] Assume $TA$ exists. Then it is equal to the common theory of all 
$\kappa$-generic $(M,\sigma)$, and a model of $TA$ is $\kappa$-saturated if and only if it is $\kappa$-generic.
\end{itemize} 

Let $(A,\sigma)\in\mathcal{C}_{T,\sigma}$. Put $A':=\PP\cap A$ and $\sigma':=\sigma\!\upharpoonright_{A'}$. We claim that any extension $(A',\sigma')\subseteq(B',\tau')$ in $\mathcal{C}_{T^{\PP},\sigma}$ 
is of the form $(B\cap\PP,\tau\!\upharpoonright_{B\cap\PP})$ for some extension $(B,\tau)\supseteq(A,\sigma)$ in $\mathcal{C}_{T,\sigma}$. Indeed, we may suppose that $A$ and $B'$ are 
subsets of some $M\models T$. Since $\PP$ is stably embedded in $T$ and $\PP=\PP^{\eq}$, we have that $\tp(A/A')\vdash\tp(A/\PP(M)$ and so in particular $\tp(A/A')\vdash\tp(A/B')$. It follows that $\sigma\cup\tau'$ is an 
elementary permutation of $AB'$ and thus extends to an elementary permutation $\tau$ of $B:=\acl(AB')$, proving the claim.

\begin{claim}\label{cl:k-gen}
Suppose $(M,\sigma)\models T_\sigma$ is $\kappa$-generic for $\mathcal{C}_{T,\sigma}$ and $\kappa$-saturated. Then $(\PP(M),\sigma)\models T^{\PP}_\sigma$ is $\kappa$-generic for $\mathcal{C}_{T^{\PP},\sigma}$ and 
$\kappa$-saturated.
\end{claim}

\begin{proof}[Proof of Claim \ref{cl:k-gen}]
Since $(\PP(M),\sigma)$ is interpretable in $(M,\sigma)$, $\kappa$-saturation of the structure $(\PP(M),\sigma)$ is clear. We now prove $\kappa$-genericity. 
Suppose $(A',\sigma')\subseteq(B',\tau')$ is an extension in $\mathcal{C}_{T^{\PP},\sigma}$ and $f':(A',\sigma')\hookrightarrow(\PP(M),\sigma)$ is an embedding. We may identify $A'$ with its image 
in $M$ and put $A:=\acl(A')\subseteq M$. Then $A$ is $\sigma$-invariant, and we have $(A,\sigma)\in\mathcal{C}_{T,\sigma}$, $A'=A\cap\PP$ and $\sigma'=\sigma\!\upharpoonright_{A'}$. Let 
$(B,\tau)\supseteq(A,\sigma)$ be as provided by the claim in the previous paragraph. By $\kappa$-genericity of $(M,\sigma)$, there is an embedding $g:(B,\tau)\hookrightarrow(M,\sigma)$ over $A$. We may thus 
define $g':=g\!\upharpoonright_{B'}:(B',\tau')\hookrightarrow(\PP(M),\sigma)$, an embedding extending $f'$ as desired.
\end{proof}

Now any $(P,\sigma)\models T^{\PP}_\sigma$ embeds into some $(M_0,\sigma_0)\models T_\s$ (since $\sigma$ is an elementary map in the sense of the original language $L$), so by (1) into some $(M,\sigma)\models T_\sigma$ which is $\kappa$-generic for $\mathcal{C}_{T,\sigma}$ and $\kappa$-saturated. Thus $(P,\sigma)\subseteq(\PP(M),\sigma)$ with $(\PP(M),\sigma)$ $\kappa$-generic for $\mathcal{C}_{T^{\PP},\sigma}$ and 
$\kappa$-saturated by Claim \ref{cl:k-gen}. Using (1) again, this shows the existence of $T^{\PP}A$, proving (i).

\smallskip

To prove (ii), replacing $(M,\sigma)$ by an elementary extension, we may assume that $(M,\sigma)$ is $\kappa$-saturated, so $\kappa$-generic by (2). By Claim \ref{cl:k-gen}, $(\PP(M),\sigma)\models T^{\PP}_\sigma$ is $\kappa$-generic, so in particular a model of $T^{\PP}A$.

\smallskip

Part (iii) follows from the fact that every $(P,\sigma)\models T^{\PP}A$ embeds into some $(M,\sigma)\models TA$. By (ii), $(\PP(M),\sigma)\models T^{\PP}A$, and so $(\PP(M),\sigma)\elres(P,\sigma)$ by model-completeness of $T^{\PP}A$.
\end{proof}

\bigskip
\section{Complex analytic jet spaces}

\noindent
In this appendix we review a particular construction of higher order tangent spaces in complex analytic geometry.
The exposition given here is based on the third author's (more detailed) unpublished survey~\cite{moosajet} of related constructions and their use in model theory.

\subsection{Linear spaces}
Suppose $X$ is a complex variety.
A {\em linear space over $X$} is a complex variety over~$X$, $L\to X$, whose fibres are uniformly equipped with complex vector space structure.
That is, there are holomorphic maps for addition $+:L\times_XL\to L$, scalar multiplication $\lambda:\mathbb C\times L\to L$, and zero section $z:X\to L$, all over $X$, satisfying the usual axioms.
For example, $X\times \mathbb C^n$ is a linear space over~$X$, it is  the {\em trivial} linear space of rank $n$.

There is a natural notion of {\em homomorphism} between linear spaces $L$ and $L'$ over~$X$, the set of which is denoted by $\hom_X(L,L')$ and has canonically the structure of an $\mathcal O_X(X)$-module.
See~\cite[$\S$1.4]{fischer76} for details.

Given a coherent analytic sheaf $\mathcal F$ on $X$, there is a {\em linear space over $X$ associated to $\mathcal F$}, denoted by $L(\mathcal F)\to X$ which has the property that for all open $U$ in $X$
\begin{equation*}
\mathcal F(U)\isom \hom_U(L(\mathcal F)_U,U\times\mathbb C)
\end{equation*}
as $\mathcal O_X(U)$-modules.
Locally $L(\mathcal F)\to X$ can be described as follows: let $U\subset X$ be a (euclidean) open subset for which there exists a resolution of $\mathcal F_U$ as
$$\xymatrix{
\mathcal O^p_U\ar[r]^\alpha & \mathcal O_U^q\ar[r] & \mathcal F_U\ar[r] & 0}$$
Represent $\alpha$ as a $q\times p$ matrix $M=(m_{ij}(x))$ with entries in $\mathcal O_U$.
Then $L(\mathcal F)_U$ is the subspace of the trivial linear space $U\times\mathbb C^q$ defined by the equations
$$m_{1i}(x)y_1+\cdots+m_{qi}(x)y_q=0$$
for $i=1,\dots,p$.

\begin{fact}[Section~1.8 of~\cite{fischer76}]
\label{linspace}
If $\mathcal F$ is a coherent analytic sheaf on $X$ and $x\in X$ then there is a canonical isomorphism $L(\mathcal F)_x\isom\hom_{\mathbb C}(\mathcal F_x\otimes_{\mathcal O_{X,x}}\mathbb C,\mathbb C)$
\end{fact}

What about definability of these objects in $\ccm$?
We can relativise the usual compactification of $\mathbb C^n$ as $\mathbb P_n(\CC)$ as follows.
If in the above local construction of $L(\mathcal F)$ we treat $(y_1:\dots:y_q)$ as homogeneous co-ordinates, then the equations cut out a complex analytic subset of $U\times\mathbb P_{q-1}(\mathbb C)$, and one obtains by gluing what is called the {\em projective linear space associated to $\mathcal F$}, denoted by $\mathbb P(\mathcal F)\to X$.
Applying the construction to the sheaf $\mathcal F\times\mathcal O_X$ we see that the linear space $L(\mathcal F)\to X$ embeds as a Zariski open subset of the projective linear space $\mathbb P(\mathcal F\times\mathcal O_X)\to X$ in such a way that linear structure  extends meromorphically to the projective linear space.
In particular, if $X$ is compact then $L(\mathcal F)\to X$, together with the uniform vector space structure on the fibres, is definable in $\ccm$.

\medskip
\subsection{Relative differentials}
Suppose $\pi:X\to S$ is a holomorphic map of complex varieties and consider the diagonal embedding $d:X\to X\times_SX$.
Recall that the inverse image sheaf $d^{-1}\mathcal O_{X\times_SX}$ is by definition the sheaf of rings on $X$ which assigns to any open set $U$ the direct limit of the rings $O_{X\times_SX}(V)$ as $V$ ranges among open neighbourhoods of $d(U)$ in $X\times_SX$.
There is an induced homomorphism of sheaves $d^{-1}\mathcal O_{X\times_SX}\to\mathcal O_X$ which is surjective because $d$ is a closed immersion.
Let $\mathcal I$ be the (ideal sheaf) kernel of this homomorphism.
The {\em sheaf of relative $n$-differentials} is the sheaf $\underline\Omega_{X/S}^n:=\mathcal I/\mathcal I^{n+1}$, viewed as a sheaf of $\mathcal O_X$-modules using the first co-ordinate projection.

\begin{fact}[Corollaire~2.7 of~\cite{groth}]
\label{diffgerm}
If $\pi:X\to S$ is a holomorphic map of complex varieties and $x\in X$ then there is a canonical isomorphism
$$\underline\Omega_{X/S,x}^n\otimes_{\mathcal O_{X,x}}\CC\isom\mathfrak m_{X_{\pi(x)},x}/\mathfrak m_{X_{\pi(x)},x}^{n+1}$$
where $\mathfrak m_{X_{\pi(x)},x}$ is the maximal ideal of the local ring of the fibre $X_{\pi(x)}$ at the point~$x$.
\end{fact}

\medskip
\subsection{Relative jet spaces}
The word ``jet'' is used in various ways in the literature.
The reader should be warned that our jet spaces are different from those of Buium, which among model theorists usually go by the term ``prolongation spaces''.

Given a point $y$ on a complex variety $Y$, by the {\em $n$th jet space of $Y$ at $y$} we mean the finite dimensional complex vector space $\jet^n(Y)_y:=\hom_{\mathbb C}(\mathfrak m_{Y,y}/\mathfrak m_{Y,y}^{n+1},\mathbb C)$.
We view this as a kind of higher order tangent space.

\begin{lemma}
\label{jetsdetermine}
Suppose $A$ and $B$ are irreducible complex analytic subsets of a complex variety $Y$ and $y\in A\cap B$.
If $\jet^n(A)_y=\jet^n(B)_y$ for all $n\geq1$, then $A=B$.
\end{lemma}

\begin{proof}
Note that in the statement we are viewing $\jet^n(A)_y$ and $\jet^n(B)_y$ canonically as subspaces of $\jet^n(Y)_y$.
We will show that $A\subseteq B$ and conclude by symmetry that $A=B$.
As $A\cap B$ is an analytic subset of $A$, and as $A$ is irreducible, it suffices to show that in some non-empty (euclidean) open subset of $Y$ containing $y$, say $U$, $(A\cap U)\subseteq(B\cap U)$.
This in turn reduces to showing that the defining ideal of the germ of $B$ at $y$ is contained in that of the germ
of $A$.

To that end, suppose $f\in \mathcal O_{Y,y}$ is a germ of a holomorphic function at $y$ which vanishes on the germ of $B$.
It follows that every linear functional in $\jet^n(B)_y$ vanishes on $f$, and so as $\jet^n(A)_y=\jet^n(B)_y$, the same is true of all functionals in $\jet^n(A)_y$.
This implies that the image of $f$  in $\mathcal O_{A,y}$ is contained in $\mathfrak m_{A,y}^{n+1}$ for all~$n$.
As $\bigcap_n\mathfrak m_{A,y}^{n+1}=0$, we have that the image of $f$ in $\mathcal O_{A,y}$ is zero.
\end{proof}

What is particularly important for us is that when $\pi:X\to S$ is a holomorphic map of complex varieties, the $n$th jet spaces of the fibres of $\pi$ vary uniformly.

\begin{definition}
Suppose $\pi:X\to S$ is a holomorphic map of complex varieties.
By the {\em $n$th jet space of $X$ relative to $S$}, which we will denote by $\jet^n (X/S)\to X$, we will mean the linear space over $X$ associated to $\underline\Omega_{X/S}^n$.
\end{definition}

A consequence of Facts~\ref{linspace} and~\ref{diffgerm} is that the fibres of $\jet^n (X/S)\to X$ are
$$\jet^n(X/S)_x=\jet^n(X_{\pi(x)})_x=\hom_{\mathbb C}(\mathfrak m_{X_{\pi(x)},x}/\mathfrak m_{X_{\pi(x)},x}^{n+1},\mathbb C)$$
for all $x\in X$.
Moreover, it follows from our discussion of linear spaces and their compactification in projective linear spaces that when $X$ and $S$ are compact, $\jet^n (X/S)\to X$ as well as the vector space structure on the fibres is (uniformly) definable in $\ccm$.

\bigskip
\section{The abelian group configuration}

\noindent
The abelian group configuration theorem, Theorem~\ref{T:AbGpConf} below, is part of
the folklore. A version of the statement appears for example in~\cite[Section 4]{evhr}, but we could find no complete published proof, so we supply
one.

\begin{lemma}\label{L:Cross-comm}
    Let $G$ be an $\infty$-definable connected group in a stable
    theory which eliminates imaginaries.
    Let $p$ be the generic type of $G$. Suppose that for $(a,b,c) \models  p^{(3)}$,
	\[ abc \in \acl(b,\acl(a,c)\cap\acl(b,abc)) .\]
    We will denote this situation by 
    
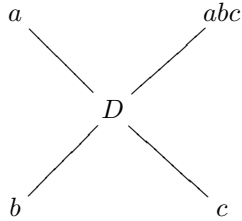
\[ \xymatrix{
      a \ar@{-}'[rd][rrdd]  &   & abc\\
                            & D & \\
      b \ar@{-}'[ur][uurr]  &   & c
      }
      \]
      where $D := \acl(a,c) \cap \acl(b,abc)$.
      Then $G$ is abelian.

\end{lemma}
\begin{proof}
  We first show that for $(e,b) \models p^{(2)}$, we have $b^{-1}eb \in \acl(e)$.

  Let $(a',c') \models  \tp((a,c) / D) |_{\acl(a,b,c)}$.
  Note first that $a',c' \ind_{\acl(a,c)} b$ since $D \subseteq  \acl(a,c)$,
  so by transitivity
  \begin{equation}\label{bind} a,c,a',c' \ind b .\end{equation}

  Meanwhile, by stationarity and independence, $(a',c') \equiv _{\acl(D,b)} (a,c)$.

  By assumption $b,abc \in \acl(D,b)$, so it follows that
    \[ a'bc' = abc .\]

  Let $e := a^{-1}a'$, so $b^{-1}eb = cc'^{-1}$. Then by (\ref{bind}), $b^{-1}eb \ind_e b$, so $b^{-1}eb \in \acl(e)$.

  Meanwhile, $a\ind_b bc$, and so since $bc\models p|_b$, we have 
  (by \cite[Lemma~1.6.9]{Pil96})
  that $a\ind_b abc$. Hence $a\ind b,abc$. So $a\ind
  D$. Since also $a\ind_D a'$, it follows that $a\ind a'$, and so $e\models p$.
  Applying (\ref{bind}) once again, we conclude that $(e,b) \models p^{(2)}$.

  It follows that the action of $G$ on itself by conjugation is generically
  constant, and hence trivial, so $G$ is abelian.
  Explicitly: if $(a,b)\models p^{(2)}$,
  we may write $b$ as $cd^{-1}$ with $(a,c,d)\models p^{(3)}$;
  we then conclude from the above that $c^{-1}ac$ and $d^{-1}ad$ are in $\acl(a)$;
  but $c \equiv_{\acl(a)} d$ by stationarity of $p$, so $c^{-1}ac = d^{-1}ad$.
  So $b^{-1}ab=a$, and so the action of $b$ by conjugation is generically
  trivial, and hence trivial. Now any element $g$ of $G$ is a product of two
  generics, and so conjugation by $g$ is trivial.
\end{proof}

\begin{theorem}[Abelian Group Configuration]\label{T:AbGpConf}
  Let $T=T^{\eq}$ be stable, let $M \models  T$ be $|T|^+$-saturated, and suppose in
  $\UU \elres M$ we have the following dependence diagram
      
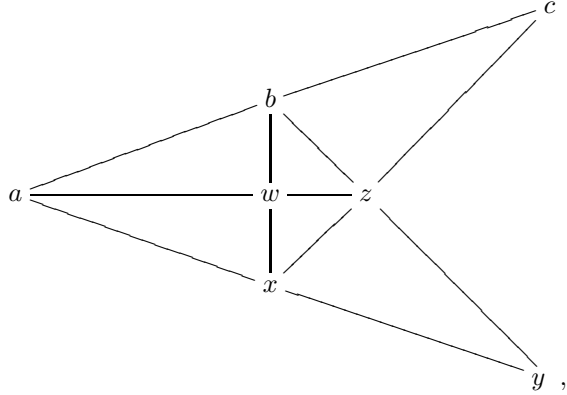
\[ \xymatrix{
 & & &&&& c \ar@{-}'[ddll][dddlll] \ar@{-}'[dlll][ddllllll]   \\
	& & & b \ar@{-}'[d][dd] & & & \\
	a & &  & w & z \ar@{-}'[l][llll] & & \\
	&  & & x &&   & \\
	&& & & & & y\;\; , \ar@{-}'[uull][uuulll] \ar@{-}'[ulll][uullllll] } \]
    where each intersection point corresponds to an element of $\UU$,
    if $a_1,a_2,a_3$ are the three distinct points on a line then $a_i \in
    \acl(M,a_j,a_k)$ whenever $\{i,j,k\}=\{1,2,3\}$, and any triple which is
    not on a common line, and is not\footnote{This exception was erroneously
    omitted in the published version of this article.}
    the triple $(w,c,y)$, is independent over $M$.

    Then there is a $\infty$-definable connected {\bf abelian} group $G$ defined
    over $M$, and $a',b',c',x',y',z'\in G$ such that each primed element is
    interalgebraic over $M$ with the corresponding unprimed element, and such
    that $c'=b'a'$, $y'=a'x'$ and $z'=b'y' = c'x'$.
\end{theorem}

\begin{proof}
  By the usual group configuration theorem, \cite[Theorem 5.4.5, Remark
  5.4.9]{Pil96}, 
  we obtain everything but commutativity of $G$. 
  On the other hand, note that we have the configuration
      
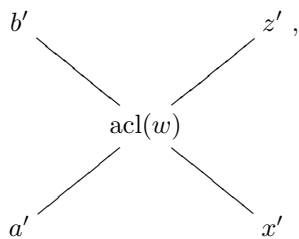
\[ \xymatrix{
      b' \ar@{-}'[rd][rrdd]  &   & z'\\
                            & \acl(w) & \\
      a' \ar@{-}'[ur][uurr]  &   & x'
      } ,\]
  and so $G$ is abelian by Lemma \ref{L:Cross-comm}.
\end{proof}


\end{document}